\documentclass{article}
\usepackage{sectsty}

\sectionfont{\large}
\subsectionfont{\normalsize}

\usepackage{ntimes,amsmath,amssymb,amsthm}
\newtheorem{theorem}{Theorem}
\newtheorem{lemma}[theorem]{Lemma}
\newtheorem{remark}[theorem]{Remark}
\newtheorem{corollary}[theorem]{Corollary}
\newtheorem{proposition}[theorem]{Proposition}
\newcommand{\E}{\mathbb{E}}
\newcommand{\F}{\mathbb{F}}
\newcommand{\R}{\mathbb{R}}

\newcommand{\N}{\mathbb{N}}
\newcommand{\U}{\mathbb{U}}
\newcommand{\Div}{\mathop{\mbox{\rm div}}}
\title{Differentiability results and sensitivity calculation for optimal control of incompressible two-phase Navier-Stokes equations with surface tension}
\author{Elisabeth Diehl\thanks{Dep. of Mathematics, TU Darmstadt,
Dolivostr. 15, 64293 Darmstadt, Germany,
diehl@mathematik.tu-darmstadt.de},
Johannes Haubner\thanks{Dep. of Mathematics, TU M\"unchen, Boltzmannstr. 3, 85748 Garching b.
M\"unchen, haubnerj@ma.tum.de}, Michael Ulbrich\thanks{Dep. of Mathematics, TU M\"unchen,
Boltzmannstr. 3, 85748 Garching b. M\"unchen, mulbrich@ma.tum.de},
Stefan Ulbrich\thanks{Dep. of Mathematics, TU Darmstadt,
Dolivostr. 15, 64293 Darmstadt, Germany, stefan.ulbrich@tu-darmstadt.de}}
\begin{document}
\maketitle
\abstract{We analyze optimal control problems for two-phase Navier-Stokes equations with surface
tension. Based on $L_p$-maximal regularity of the underlying linear problem and
recent well-posedness results of the problem for sufficiently small data we
show the differentiability of the solution with respect to initial and
distributed controls for appropriate spaces resulting form the $L_p$-maximal regularity
setting. We consider first a formulation where the interface
is transformed to a hyperplane. Then we deduce differentiability results for
the solution in the physical coordinates. Finally, we state an equivalent
Volume-of-Fluid type formulation and use the obtained differentiability
results to derive rigorosly the corresponding sensitivity equations of the Volume-of-Fluid type
formulation. For objective functionals involving the velocity field or the
discontinuous pressure or phase indciator field we derive differentiability
results with respect to controls and state formulas for the derivative.
The results of the paper form an analytical foundation for stating optimality
conditions, justifying the application of derivative based optimization methods
and for studying the
convergence of discrete sensitivity schemes based on Volume-of-Fluid
discretizations for optimal control of two-phase Navier-Stokes equations.}
\smallskip

{\bf Keywords.~ }{two-phase flow, surface tension, sharp interface, Navier-Stokes
equations, volume of fluid, differentiability, optimal control}
{\large \section{Introduction}}
We consider the incompressible sharp interface two-phase Navier-Stokes
equations. To this end, let 
the hypersurface (interface) $\Gamma(t)$ divide $\R^{n+1}$ into two open domains
$\Omega_1(t)$ and $\Omega_2(t)=\R^{n+1}\setminus \overline{\Omega_1}(t)$, $i=1,2$,
occupied by two viscous
incompressible immiscible capillary Newtonian fluids with constant densities
$\rho_i>0$ and constant viscosities $\mu_i>0$, $i=1,2$. We set
\[
 \Omega(t):=\Omega_1(t)\cup \Omega_2(t)
\]
and with the indicator functions $1_{\Omega_i}$
\[
  \rho=\rho_1 1_{\Omega_1}+\rho_2 1_{\Omega_2},\quad
  \mu=\mu_1 1_{\Omega_1}+\mu_2 1_{\Omega_2}.
\]
Moreover, we denote by $\nu(t,\cdot)$ the normal field on $\Gamma(t)$
pointing form $\Omega_1(t)$ to $\Omega_2(t)$, by $V(t,\cdot)$ the normal velocity
of the interface $\Gamma(t)$ and by $\kappa(t,\cdot)$ the mean curvature of $\Gamma(t)$ with respect to
$\nu(t,\cdot)$. Then $\kappa(t,x)$ is negative when $\Omega_1(t)$ is convex
close to $x\in \Gamma(t)$ and is for sufficiently smooth $\Gamma(t)$
given by
\[
 \kappa(t,\cdot)=-\Div\nolimits_\Gamma \nu(t,\cdot)
\]
(note that this coincides with $-\Div \nu(t,\cdot)$ if $\nu(t,\cdot)$
admits a differentiable extension to a neighborhood of $\Gamma(t)$).
Finally, if $v$ is defined and admits boundary traces on both domains
$\Omega_i(t)$ then
\[
 [v]=(v|_{\Omega_2(t)}-v|_{\Omega_1(t)})|_{\Gamma(t)}
\]
denotes the jump of $v$ accross $\Gamma(t)$.
The two-phase Navier-Stokes equations with surface tension
then read
\begin{align}\label{P}
\begin{aligned}
 \rho (\partial_t u+u\cdot\nabla u)-\mu \Delta u+\nabla q &=c &&
\mbox{in $\Omega(t)$,}\\
\Div u&=0 &&\mbox{in $\Omega(t)$,}\\
 -[S(u,q;\mu) \nu]&=\sigma\kappa \nu && \mbox{on $\Gamma(t)$,}\\
 [u]&=0 && \mbox{on $\Gamma(t)$,}\\
 V&=u^\top \nu &&\mbox{on $\Gamma(t)$,}\\
u(0)&=u_0&&\mbox{on $\Omega(0)$,}\\ 
 \Gamma(0)&=\Gamma_0.
\end{aligned}
\end{align}
with the stress tensor $S(u,q;\mu)=-qI+\mu (\nabla u+\nabla u^\top)$
and the surface tension coefficient $\sigma>0$.
Here, $c$ denotes some control.

The conditions on the interface ensure 
that the surface tension balances the jump of the normal stress on the interface
the balance of surface tension and the jump of the normal stress on the
interface, the continuity of the velocity across the interface and the
transport of the interface by the fluid velocity.

We note that the first four equations can be written in weak form on the
whole domain by
\begin{align}\label{weak1}
&\int_{\R^{n+1}} \left(\partial_t (\rho u)+\Div(\rho u\otimes u)-c)^\top \varphi+S(u,q;\mu):\nabla\varphi\right)\,dx
=\int_{\Gamma(t)} \sigma\kappa \nu^\top
\varphi\,dS(x)\\\notag
&\hspace{.6\textwidth}\forall\,\varphi\in C_c^1(\R^{n+1};\R^{n+1}),\\\label{weak2}
&\int_{\R^{n+1}} \Div(u) \,\psi\,dx=0\quad\forall\,\psi\in C_c^1(\R^{n+1}).
\end{align}

Our aim is to study the differentiability properties of local solutions
with respect to $u_0$ and $c$. To this end, we will
work in an $L_p$-maximal regularity setting proposed in
\cite{PruessSimonett1}, see also \cite{Koehneetal, PruessSimonett2}.

There exist several papers on the existence and uniqueness of local
solutions for \eqref{P}. In \cite{Denisova, Denisova1, Shibata, Tanaka}
Lagrangian coordinates are used to obtain local well-posedness. Since this
approach makes it difficult to establish smoothing of the unknown interface,
\cite{Koehneetal, PruessSimonett1, PruessSimonett2} use a transformation to a
fixed domain and are then able to show  local well-posedness in an $L_p$
maximum regularity setting for the case $c=0$
\cite{Koehneetal, PruessSimonett1} or for the case
of gravitation \cite{PruessSimonett2}. Moreover,
they prove that the interface as well as the
solution become instantaneously real analytic. Since we are
considering a distributed control $c$ of limited regularity, the instant
analyticity is in general lost.

While optimal control problems for the Navier-Stokes equations have been
studied by many researchers, see for
example \cite{Fursikov,Gunzburger,HinzeKunisch,UlbrichNS}, there are only a few
contributions in the context of two-phase Navier-Stokes equations, mainly
for phase-field formulations with semidiscretization in time. In
\cite{Hintermueller1} optimal boundary control of a time-discrete
Cahn-Hilliard-Navier-Stokes system with matched densities is studied. By using regularization
techniques, existence of optimal solutions and optimality conditions are
derived. Analogous results for distributed optimal control with unmatched
densities for the diffuse interface model of \cite{Abelsetal} have been
obtained in \cite{Hintermueller2}. Using the
same model, \cite{Garckeetal1} derive based on the stable time discretization proposed in
\cite{Garckediscr} necessary optimality conditions
for the time-discrete and the fully discrete optimal control problem. Moreover,
the differentiability of the control-to-state mapping
for the semidiscrete problem is shown. Optimal control of a binary fluid
described by its density distribution, but without surface tension, is
studied in \cite{Banas}. Different numerical approaches for the optimal
control of two-phase flows are discussed in \cite{Braack}.

In this paper we derive differentiability results of the solution of the
two-phase Navier-Stokes equations \eqref{P} with respect to controls. The
results can be used to state optimality conditions and to justify the
application of derivative based optimization methods.
To the best of
our knowledge, this is the first work providing differentiability properties
of control-to-state mappings for sharp interface models of two-phase
Navier-Stokes flow. The analysis is involved, since
the moving interface renders a variational analysis difficult. Therefore it
is beneficial, to first consider a transformed problem with fixed interface.
However, since most numerical approaches are working in physical coordinates,
we derive also differentiability results for the original problem. Since the
normal derivative of the velocity is in general discontinuous at the interface,
the sensitivities of the velocity are discontinuous across the interface.
Moreover, the pressure is in general discontinuous at the interface and
thus differentiability properties with respect to controls in strong
spaces hold only away from the interface while at the interface
differentiability properties can only be expected in the weak topology of
measures. The same applies to phase indicators which are often used in
Volume-of-Fluid (VoF)-type approaches.
In order to obtain a PDE-formulation for the sensitivity equations, we work
with a Volume-of-Fluid (VoF)-type formulation based on a discontinuous phase
indicator and derive carefully a corresponding sensitivity equation.

We build on the quite recent existence and uniqueness results obtained for sufficiently
small data by \cite{PruessSimonett1}, see also \cite{Koehneetal, PruessSimonett2}.
We consider first a formulation, where the interface is transformed to a
hyperplane. By using $L_p$-maximal regularity of a linear system and
applying a refined version of a fixed point theorem, we show
differentiability of the transformed state with respect to controls in the
maximum regularity spaces. A similar technique was recently used in
\cite{HUU} to show differentiability properties for shape optimization of
fluid-structure interaction, but the analysis of the fixed
point iteration is very different from two-phase flows considered here.
In fact, the main difficulties in fluid-structure interaction arise from the
coupling of a hyperbolic equation for the solid with the Navier-Stokes
equations for the fluid while in two phase flows the moving interface and
the surface tension are the main challenge.
In a second step we deduce differentiability results for
the control-to-state map in the physical coordinates. Finally, we derive an
equivalent Volume-of-Fluid (VoF)-type formulation based on a discontinuous phase
indicator that is governed by a multidimensional transport equation. By
using the obtained differentiability results, we are able to justify a
sensitivity system for the VoF-type formulation, which invokes
measure-valued solutions of the linearized transport equation. This can be
used as an analytical foundation to study the convergence of discrete
sensitivity schemes for VoF-type methods. Moreover, we obtain the
differentiability of objective functionals invoking the velocity field or
the discontinuous pressure or phase indicator field and state formulas for
the derivative.

The paper is organized as follows. In section \ref{sec:2}, the transformed
problem is formulated. In section \ref{sec:3}, existence, uniqueness and
differentiability of the control-to-state mapping is shown. The analysis
starts in \ref{sec:31} for the transformed problem with flat interface. In
\ref{sec:32} differentiability results for the original problem in physical
coordinates are derived. In \ref{sec:33} the VoF-type formulation and its
sensitivity equation are justified. In section \ref{sec:4} we derive some
analytical settings for the application of optimization methods.
In \ref{sec:41} we consider objective functionals involving the velocity
field and state differentiability results. Subsequently, we discuss
in \ref{sec:42} objective functionals involving the pressure field or the
phase indicator, obtain their differentiability with respect to controls
as well as a formula for the derivative.

\section{Transformation to a flat interface}\label{sec:2}
In this paper, we consider as in Pr\"uss and Simonett
\cite{PruessSimonett1} the problem in
$n+1$ dimensions, where $\Gamma_0$ is the graph of a sufficiently smooth function
$h_0:\R^n\to\R$, i.e.,
\begin{align*}
  \Gamma_0 &=\{(x,y)\in\R^n\times \R: y=h_0(x)\},\\
  \Omega_1(0)&=\{(x,y)\in\R^n\times \R: y<h_0(x)\},\\
  \Omega_2(0)&=\{(x,y)\in\R^n\times \R: y>h_0(x)\}.
\end{align*}
The interface has then the form
\[
 \Gamma(t)=\{(x,h(t,x)):\, x\in \R^n\},
\]
where $h: [0,t_0]\times \R^n\to\R$ with $h(0,\cdot)=h_0$ and $t_0>0$ is some
final time.
We note that the case of bounded fluid domains is considered in
\cite{Koehneetal}. The analysis of this paper should also extend to this
setting, but the presentation would be more technical.

If $h(t,\cdot)$ has second derivatives then normal and curvature of the interface $\Gamma(t)$ are
given by
\begin{equation}\label{nukappa}
\begin{split}
\hat\nu(t,x)=\nu(t,x,h(t,x))&=\frac{1}{\sqrt{1+|\nabla h(t,x)|^2}}\binom{-\nabla
h(t,x)}{1},\\
\hat\kappa(t,x)=\kappa(t,x,h(t,x))&=\Div\nolimits_x \left(\frac{\nabla h(t,x)}{\sqrt{1+|\nabla h(t,x)|^2}}\right)
=\Delta h-G_\kappa(h),
\end{split}
\end{equation}
where $\nabla h$ and $\Delta h$ denote the gradient and Laplacian of $h$ with
respect to $x$ and
\begin{align}\label{Gkadef}
G_\kappa(h)=\frac{|\nabla h|^2 \Delta h}{(1+\sqrt{1+|\nabla h|^2})\sqrt{1+|\nabla h|^2}}+
\frac{\nabla h^\top \nabla^2 h  \nabla h}{(1+|\nabla h|^2)^{3/2}}.
\end{align}
Following
\cite{PruessSimonett1}, we now transform the problem to $\dot\R^{n+1}=\{(x,y)\in \R^{n+1}: y\ne 0\}$
with a flat interface at $y=0$ by using the transformation
\begin{equation}\label{trafo}
  \hat u(t,x,y)=\binom{v(t,x,y)}{w(t,x,y)}:=u(t,x,h(t,x)+y),
 \pi(t,x,y):=q(t,x,h(t,x)+y).
\end{equation}
Analogously, let with $\R_\pm^{n+1}=\{(x,y)\in \R^n\times\R:\, \pm y>0\}$
\begin{align*}
 \hat\rho(t,x,y)&=\rho(t,x,h(t,x)+y)=\chi_{\R_-^{n+1}}(x,y)
\rho_1+\chi_{\R_+^{n+1}}(x,y)\rho_2,\\
 \hat\mu(t,x,y)&=\mu(t,x,h(t,x)+y)=\chi_{\R_-^{n+1}}(x,y) \mu_1+\chi_{\R_+^{n+1}}(x,y) \mu_2.
\end{align*}
As in \cite{PruessSimonett1}, we work with the following function spaces.
Let $\Omega\subset\R^m$ be open and $X$ be a Banach space. $L_p(\Omega;X)$,
$H_p^s(\Omega;X)$, $1\le p\le\infty$, $s\in\R$,
denote the X-valued Lebesgue and Bessel potential spaces of order s, respectively.
We note that $H_p^k(\Omega;X)=W_p^k(\Omega;X)$ for $k\in\N_0$, $1<p<\infty$
with the Sobolev-Slobodetski\v{i} spaces $W_p^k$. Moreover, we will use
the fractional Sobolev-Slobodetski\v{i} spaces
$W_p^s(\Omega;X)$, $1\le p<\infty$, $s\in (0,\infty)\setminus\N$, 
with norm
\[
 \|g\|_{W_p^s(\Omega;X)}=\|g\|_{W_p^{[s]}(\Omega;X)}
+\sum_{|\alpha|=[s]}\left(\int_\Omega\int_\Omega
\frac{\|\partial^\alpha g(x)-\partial^\alpha
g(y)\|_X^p}{|x-y|^{m+(s-[s])p}}\,dx\,dy\right)^{1/p}
\]
We recall that $W_p^s(\Omega;X)=B_{pp}^s(\Omega;X)$ for $s\in (0,\infty)\setminus\N$
with the Besov space $B_{pp}^s$. Finally, the homogeneous Sobolev space
$\dot H_p^1(\Omega)$ is defined by
\[
 \dot H_p^1(\Omega):=(\{g\in L_{1,loc}(\Omega):
\|\nabla g\|_{L_p(\Omega)}<\infty\},\|\cdot\|_{\dot H_p^1(\Omega)}),\quad
\|g\|_{\dot H_p^1(\Omega)}:=\|\nabla g\|_{L_p(\Omega;\R^m)}.
\]
Then $\dot H_p^1(\Omega)$ is for connected $\Omega$
a Banach space if we factor out the constant functions
and equip the resulting space with the corresponding quotient norm.

Finally, for $\Omega\subset\R^m$ open or closed we denote by $BUC(\Omega;X)$ and $BC(\Omega;X)$ the space of
bounded uniformly continuous and the space of bounded continuous functions
equipped with the supremum norm, respectively. Analogously, $BUC^k(\Omega;X)$ and
$BC^k(\Omega;X)$, $k\in\N_0$, are defined for $k$-times continuously
differentiable functions with bounded uniformly continuous or bounded
continuous derivatives up to order $k$. If $\Omega$ is compact,
we may briefly write $C^k(\Omega;X)$, since boundedness und uniform continuity
are automatically satisfied.

To state the transformed problem, we follow \cite{PruessSimonett1} and
we use a fixed point formulation consisting of a linearized Stokes
problem with nonlinear right hand side. In fact, denote by
\begin{equation}\label{stokesL}
    L (\hat u,\pi,r,h)=(f,f_d,g_v,g_w,g_h),~
    (\hat u(0),h(0))=(\hat u_0,h_0),~ (\hat u,\pi,r,h)\in \E(t_0)
\end{equation}
(i.e., $r=[\pi]$ by the definition of $\E(t_0)$) the Stokes problem with free boundary
\begin{align}\label{stokes}
\begin{aligned}
\hat\rho \partial_t \hat u-\hat\mu \Delta \hat u+\nabla \pi &=f && \mbox{in $\dot\R^{n+1}$},\\
\Div \hat u & = f_d && \mbox{in $\dot\R^{n+1}$},\\
- [\hat\mu \partial_y v]-[\hat\mu\nabla_x w] &=g_v&& \mbox{on $\R^{n}$},\\
- 2[\hat\mu \partial_y w]+[\pi]-\sigma \Delta h &=g_w && \mbox{on $\R^{n}$},\\
 [\hat u] & = 0 && \mbox{on $\R^{n}$},\\
 \partial_t h-\gamma w &=g_h && \mbox{on $\R^{n}$},\\
 \hat u(0)=\hat u_0,\quad h(0)& =h_0.
\end{aligned}
\end{align}
for $t>0$. Here, $[\hat u]$ denotes the jump across the transformed
interface $y=0$ and $\gamma w(x)=w(x,0)$ denotes the trace of a
function $w: \dot\R^{n+1}\to\R$ at $y=0$ satisfying $[w]=0$.

Then it is shown in
\cite{PruessSimonett1} that the transformation
\eqref{trafo} leads to the following problem for
$\hat u=(v,w), \pi, h$
\begin{align}\label{Ptrans}
\begin{split}
L (\hat u,\pi,[\pi],h)&=(\hat c+F(\hat u,\pi,h),F_d(\hat u,h),G_v(\hat u,[\pi],h),G_w(\hat
u,h),H(\hat u,h)),\\
(\hat u(0),h(0))&=(\hat u_0,h_0),
\end{split}
\end{align}
where the right hand sides are given by
\begin{align}\label{RHS}
\begin{split}
F_v(v,w,\pi,h)&=\hat\mu \left(-2(\nabla h\cdot\nabla_x)\partial_y v+
|\nabla h|^2 \partial_y^2 v-\Delta h \partial_y v\right)+
\partial_y \pi\nabla h\\
&\quad+\hat\rho \left(-(v\cdot\nabla_x)v+(\nabla h^\top v)\partial_y v-
w \partial_y v\right)+\hat\rho \partial_t h \partial_y v,\\
F_w(v,w,h)&=\hat\mu \left(-2(\nabla h\cdot\nabla_x)\partial_y w+
|\nabla h|^2 \partial_y^2 w-\Delta h \partial_y w\right)\\
&\quad+\hat\rho \left(-(v\cdot\nabla_x)w+(\nabla h^\top v)\partial_y w-
w \partial_y w\right)+\hat\rho \partial_t h \partial_y w,\\
F_d(v,h) &= \nabla h^\top \partial_y v,\\
G_v(v,w,[\pi],h)&=-[\hat\mu (\nabla_x v+(\nabla_x v)^\top)]\nabla h+
|\nabla h|^2 [\hat\mu \partial_y v]+(\nabla h^\top [\hat\mu\partial_y v])\nabla h\\
&\quad-[\hat\mu\partial_y w]\nabla h+
\left([\pi]-\sigma (\Delta h-G_\kappa(h))\right)\nabla h,\\
G_w(v,w,h)&=-\nabla h^\top [\hat\mu\partial_y v]-
\nabla h^\top [\hat\mu\nabla_x w]+|\nabla h|^2 [\hat\mu \partial_y w]-\sigma
G_\kappa(h),\\
H(v,w,h) &= -(\gamma v)^\top\nabla h.
\end{split}
\end{align}
Note that all terms except $G_\kappa(h)$ 
are polynomials in $(v,w,\pi,[\pi],h)$ and derivatives of $(v,w,\pi,h)$.
Moreover, all terms are linear with respect to second derivatives and
$G_\kappa(h)$ is the pointwise superposition of a smooth function with $\nabla h$ and
$\nabla^2 h$.

\begin{remark}
The transformed version of the deformation tensor $D(u)=\nabla u+\nabla
u^\top$ is given by ${\mathcal D}(\hat u,h)={\mathcal D}(v,w,h)$, where
\begin{align*}
{\mathcal D}(\hat u,h)=\nabla \hat u+\nabla \hat u^\top-
\binom{\nabla h \partial_y \hat u^\top}{0}-\binom{\nabla h \partial_y \hat u^\top}{0}^\top.
\end{align*}
Then the compatibility condition \eqref{comp} can with
$\hat\nu(0,x)=\frac{1}{\sqrt{1+|\nabla h_0(x)|^2}}\binom{-\nabla
h_0(x)}{1}$ equivalently be written as
\begin{equation}\label{comptrans}
\begin{split}
 &[\hat\mu {\mathcal D}(\hat u_0,h_0)\hat\nu(0)-\hat\mu(\hat\nu(0)^\top 
{\mathcal D}(\hat u_0,h_0)\hat\nu(0))\hat\nu(0)]=0,\\
& \Div \hat u_0=F_d(\hat u_0,h_0),\quad [\hat u_0]=0.
\end{split}
\end{equation}
\end{remark}

\section{Well-posedness and differentiability with respect to
controls}\label{sec:3}
\subsection{Results for the transformed problem}\label{sec:31}
By applying a fixed point theorem to \eqref{Ptrans}, the following result is shown in
\cite{PruessSimonett1} for $\hat c=0$.
\begin{theorem}\label{thm:ex}
Let $p>n+3$ and consider the case $c=0$, i.e. $\hat c=0$. Let
\begin{align}\label{Udef1}
 \U_{\hat u}:=W_p^{2-2/p}(\dot \R^{n+1},\R^{n+1}),\quad
\U_h:=W_p^{3-2/p}(\R^n).
\end{align}
Then for any $t_0>0$  there exists $\hat\varepsilon_0=\hat\varepsilon_0(t_0)>0$
such that for all initial values
\[
  (\hat u_0,h_0)\in \U_{\hat u}\times \U_h
\]
satisfying, with $u_0(x,h_0(x)+y)=\hat u_0(x,y)$, the compatibility conditions
\begin{equation}\label{comp}
 [\mu D(u_0)\nu(0)-\mu(\nu(0)^\top D(u_0) \nu(0))\nu(0)]=0,\quad
 \Div u_0=0,\quad
  [u_0]=0,
\end{equation}
as well as the smallness condition
\[
\|\hat u_0\|_{\U_{\hat u}}+\|h_0\|_{\U_h}\le
\hat\varepsilon_0
\]
there exists a unique solution of the transformed problem \eqref{Ptrans} with
\[
 (\hat u,\pi,[\pi],h)\in \E(t_0),
\]
where 
with $J=(0,t_0)$
\begin{align}\notag
 \E_1(t_0)&=\{\hat u\in H_p^1(J;L_p(\R^{n+1},\R^{n+1}))\cap
L_p(J;H_p^2(\dot\R^{n+1},\R^{n+1})): [\hat u]=0\},\\
\notag
 \E_2(t_0)&=L_p(J;\dot H_p^1(\dot\R^{n+1})),\\
\label{Edef}
 \E_3(t_0)&=W_p^{1/2-1/(2p)}(J;L_p(\R^{n}))\cap
 L_p(J;W_p^{1-1/p}(\R^{n})),\\
\notag
 \E_4(t_0)&=W_p^{2-1/(2p)}(J;L_p(\R^{n}))\cap
 H_p^1(J;W_p^{2-1/p}(\R^{n}))\\
\notag
 &\quad \cap W_p^{1/2-1/(2p)}(J;H_p^2(\R^{n}))\cap
 L_p(J;W_p^{3-1/p}(\R^{n})),\\
\notag
 \E(t_0)&=\{(\hat u,\pi,r,h)\in \E_1(t_0)\times \E_2(t_0)\times \E_3(t_0)\times
 \E_4(t_0):[\pi]=r\}.
\end{align}
Moreover, $(\hat u,\pi,[\pi],h)\in \E(t_0)$ depends continuously on $(\hat
u_0,h_0)\in \U_{\hat u}\times \U_h$ satisfying \eqref{comp}.
\end{theorem}
\begin{proof}
See \cite[Thm. 6.3]{PruessSimonett1}.
\end{proof}

Our first aim is to study the differentiability properties of the control-to-state
map $(\hat u_0,\hat c)\mapsto (\hat u,\pi,[\pi],h)$. Note that we consider
also the case $\hat c\ne 0$. The proof is carried out by an appropriate
extension of the fixed point argument for \eqref{Ptrans} based on
Theorem \ref{thm:fix}.

To apply the fixed point argument,
the following $L_p$-maximum regularity result of
\cite{PruessSimonett1} for the linearized problem \eqref{stokes} will
be essential.
\begin{theorem}\label{thm:Lplin}
Let $1<p<\infty$ be fixed, $p\ne 3/2,3$ and assume that $\rho_i,\mu_i$ are
positive constants. For arbitrary $t_0>0$ let $J=(0,t_0)$ and let
$\E_1(t_0),\ldots,\E_4(t_0), \U_{\hat u},\U_h$
be defined by \eqref{Edef}, \eqref{Udef}. Set
\begin{align}\label{Fdef}
\begin{aligned}
 \F_1(t_0)&=L_p(J;L_p(\R^{n+1},\R^{n+1})),\\
 \F_2(t_0)&=H_p^1(J;\dot H_p^{-1}(\R^{n+1}))\cap L_p(J;H_p^1(\dot \R^{n+1})),\\
 \F_3(t_0)&=W_p^{1/2-1/(2p)}(J;L_p(\R^{n},\R^{n+1}))\cap
 L_p(J;W_p^{1-1/p}(\R^{n},\R^{n+1})),\\
 \F_4(t_0)&=W_p^{1-1/(2p)}(J;L_p(\R^{n}))\cap
 L_p(J;W_p^{2-1/p}(\R^{n})),\\
 \F(t_0)&=\F_1(t_0)\times \F_2(t_0)\times \F_3(t_0)\times \F_4(t_0).
\end{aligned}
\end{align}
Then for all initial values
$(\hat u_0,h_0)\in \U_{\hat u}\times \U_h$ and $(f,f_d,g,g_h)\in
\F(t_0)$
satisfying the compatibility conditions
\begin{align}\label{complin1}
 \Div \hat u_0=f_d(0) ~~ \mbox{on $\dot\R^{n+1}$},\quad
  [\hat u_0]=0 ~~ &\mbox{on $\R^n$ ~~if $p>3/2$},\\\label{complin2}
\mbox{and in addition}\quad [-\hat\mu \partial_y v_0]-[\hat\mu \nabla_x w_0]=g_v(0) ~~ & \mbox{on $\R^n$ ~~if $p>3$,}
\end{align}
there exists a unique solution $(\hat u,\pi,h)\in \E(t_0)$ of \eqref{stokes} and the solution map
\[
(f,f_d,g,g_h,\hat u_0,h_0)\in \F(t_0)\times \U_{\hat u}\times \U_h\mapsto (\hat
u,\pi,[\pi],h)\in \E(t_0)
\]
is continuous.
\end{theorem}
\begin{proof}
This follows from \cite[Thm. 5.1]{PruessSimonett1} and \cite[Lem. 6.1, (e)]{PruessSimonett1}.
\end{proof}
For homogeneous initial data we obtain immediately.
\begin{corollary}\label{cor:stokes}
Let $p>3$ and define in addition to $\E(t_0)$ and $\F(t_0)$ the spaces
\begin{align*}
 {}_0\E(t_0)&:=\{(\hat u,\pi,r,h)\in \E(t_0):~ \hat u(0)=0,~~ r(0)=0,~~ h(0)=0\},\\
 {}_0\F(t_0)&:=\{(f,f_d,g,g_h)\in \F(t_0):~ f_d(0)=0,~~ g(0)=0,~~ g_h(0)=0\}
\end{align*}
with initial value $0$ for all components that admit a trace at $t=0$.
Then \eqref{stokes} has a unique and continuous solution map
\[
(f,f_d,g,g_h,0,0)\in {}_0\F(t_0)\times \U_{\hat u}\times \U_h\mapsto (\hat
u,\pi,[\pi],h)\in {}_0\E(t_0)
\]
\end{corollary}

The fixed point argument relies on the following properties of the right
hand sides \eqref{RHS} of \eqref{Ptrans}.
\begin{lemma}\label{lem:N}
Let $p>n+3$ and set for $(\hat u,\pi,r,h)\in\E(t_0)$
\begin{equation}\label{Ndef}
 N(\hat u,\pi,r,h):=(F(\hat u,\pi,h),F_d(\hat u,h),G(\hat u,r,h),H(\hat u,h)),
\end{equation}
with $F=(F_v,F_w)$, $G=(G_v,G_w)$, $F_d$ and $H$ defined in \eqref{RHS}.
Then the mapping $N: \E(t_0)\to \F(t_0)$ is a well defined and real
analytic, more precisely,
\[
 N\in C^\omega(\E(t_0),\F(t_0)),\quad N(0)=0,\quad DN(0)=0.
\]
Moreover,
\[
 DN(\hat u,\pi,r,h)\in {\mathcal L}({}_{0} \E(t_0),{}_{0}
\F(t_0))\quad\forall\, (\hat u,\pi,r,h)\in \E(t_0).
\]
\end{lemma}
\begin{proof}
See \cite[Prop. 6.2]{PruessSimonett1}
\end{proof}
Moreover, we will need the following analogue for the spaces of the initial
values.
\begin{lemma}\label{lem:N0}
Let $p>n+3$, $\U_{\hat u}, \U_h$ be as in \eqref{Udef1} and set
\[
 \U_{\hat u,c}:=\{\hat u_0=(v_0,w_0) \in \U_{\hat u}: [\hat u_0]=0\}.
\]
Then with $G=(G_v,G_w)$ and $H$ defined in \eqref{RHS}
the mappings
\begin{align}\label{v0gh0}
(\hat u_0,h_0)\in \U_{\hat u}\times\U_h &\mapsto v_0^\top\nabla h_0\in
W^{2-2/p}_p(\dot\R^{n+1}),\\\label{H0}
(\hat u_0,h_0) \in \U_{\hat u,c}\times \U_h &\mapsto H(v_0,h_0)\in
W^{2-3/p}_p(\R^n),\\\label{G0}
(\hat u_0,r_0,h_0)\in \U_{\hat u} \times W^{1-2/p}_p(\R^n) \times
\U_h &\mapsto G(\hat u_0,r_0,h_0)\in W^{1-2/p}_p(\R^n)
\end{align}
are real analytic and the first derivatives vanish in $(\hat u_0,r_0,h_0)=0$.
\end{lemma}
\begin{proof}
Since $p>n+3$ we have $W^{1-2/p}_p(\dot\R^{n+1}) \hookrightarrow BUC(\dot\R^{n+1})$
and thus $W^{s}_p(\dot\R^{n+1})$ is a multiplication algebra, i.e. a Banach
algebra under the operation of multiplication, for all
$s\ge 1-2/p$, see e.g.
\cite[Lem. 4.1, Rem. 6.4]{Brezis}. As a consequence, \eqref{v0gh0}
is a continuous bilinear form and thus in $C^\omega(\U_{\hat u}\times\U_h,W^{2-2/p}_p(\dot\R^{n+1}))$.

Similarly, $W^{s}_p(\R^{n})$ is a multiplication algebra for all
$s\ge 1-2/p$. Since the trace
operator $\hat u_0 \in \U_{\hat u,c} \mapsto \gamma v_0\in W^{2-3/p}_p(\R^n)$
is continuous, \eqref{H0}
is a continuous bilinear form and thus real analytic.

Finally $G(\hat u_0,r_0,h_0)$ in \eqref{G0} is a polynomial in
$W^{1-2/p}_p(\R^n)$-functions and in functions of the form
$\nabla h_0 / (a+(1+\nabla h_0^\top\nabla h_0)^{k/2})$ with $a\ge 0$ and
$k\in\{1,2\}$. The function
$\Psi: v\in \R^n \mapsto v/(a+(1+v^\top v)^{k/2})$ is smooth with bounded derivatives
and $\Psi(0)=0$. Since $2-2/p>n/p$ implies
$h_0\in \U_h \mapsto \nabla h_0\in W^{2-2/p}_p(\R^n)\hookrightarrow
W^1_{(2-2/p)p}(\R^n)\cap BUC(\R^n)$, it is well known that
\begin{equation}\label{super}
h_0\in \U_h \mapsto \Psi(\nabla h_0)\in W^{2-2/p}_p(\R^n)
\end{equation}
is well defined
and continuous, see \cite[Thm. 1.1]{Brezis}. It is also differentiable. In fact, for $d\in \U_h$
\begin{align*}
 \Psi(\nabla h_0+\nabla d)-\Psi(\nabla h_0)-\Psi'(\nabla h_0) \nabla d
 =\int_0^1 (\Psi'(\nabla h_0+\tau \nabla d)-\Psi'(\nabla h_0)) \nabla d\,d\tau
\end{align*}
where the integrand is in $BUC([0,1]\times \R^n)$. Moreover, since
$v\mapsto \Psi'(\nabla v)-\Psi'(0)$ is smooth with bounded derivatives
and vanishes at $0$, the integrand is continuous from $[0,1]\to
W^{2-2/p}_p(\R^n)$ again by \cite[Thm. 1.1]{Brezis}. Hence the integral is also a Bochner integral
and thus by using the multiplication algebra property there is $C>0$ with
\begin{align*}
&\|\Psi(\nabla h_0+\nabla d)-\Psi(\nabla h_0)-\Psi'(\nabla h_0) \nabla
d\|_{W^{2-2/p}_p(\R^n)}\\
&\le C \int_0^1 \|\Psi'(\nabla h_0+\tau \nabla d)-\Psi'(\nabla h_0)\|_{W^{2-2/p}_p(\R^n)}\,d\tau
\|\nabla d\|_{W^{2-2/p}_p(\R^n)}
=o(\|d\|_{\U_h}),
\end{align*}
since $d\in \U_h \mapsto \Psi'(\nabla h_0+\nabla d)-\Psi'(0)\in W^{2-2/p}_p(\R^n;\R^n)$
is continuous at $d=0$ by \cite[Thm. 1.1]{Brezis}. Now we can iteratively show that \eqref{super} is
real analytic. In fact, we can write $d\in \U_h \mapsto \Psi'(\nabla h_0) \nabla
d=(\Psi'(\nabla h_0)-\Psi'(0)) \nabla d+\Psi'(0) \nabla d\in W^{2-2/p}_p(\R^n)$. The second term is a constant
mapping in ${\mathcal L}(\U_h,W^{2-2/p}_p(\R^n))$. Moreover, as before
$h_0\in \U_h \mapsto \Psi'(\nabla h_0)-\Psi'(0)\in W^{2-2/p}_p(\R^n;\R^n)$ is well defined
and continuous \cite[Thm. 1.1]{Brezis} and by the same arguments as above also differentiable.
Iterating the argument shows that \eqref{super} is real analytic.

We conlude that \eqref{G0}
is a polynomial in
$W^{1-2/p}_p(\R^n)$-functions and in real analytic functions of the form
\eqref{super}. Since $W^{1-2/p}_p(\R^n)$ is a multiplication algebra, we
conclude that \eqref{G0} is real analytic.

By the product structure of \eqref{v0gh0}--\eqref{G0} the
first derivatives vanish in $0$.


\end{proof}

We will work with the following extension of Banach's fixed point theorem.
\begin{theorem}\label{thm:fix}
\begin{itemize}
\item[a)] Let $U,W.Z$ be real Banach spaces, let $A\in {\mathcal L}(Z,W)$ be an
isomorphism and set $M:=\|A^{-1}\|_{{\mathcal L}(W,Z)}$.
Let $B_Z\subset Z$ be a nonempty closed convex
set and $B_U\subset U$ be a nonempty set.
Moreover, let $K: B_Z\times B_U \to W$ be Lipschitz continuous with
\[
 \|K(z,u)-K(\tilde z,\tilde u)\|_{W}\le L_z \|z-\tilde z\|_Z+L_u \|u-\tilde u\|_U
\quad\forall\, (z,u), (\tilde z,\tilde u)\in B_Z\times B_U
\]
and assume that
\begin{equation}\label{contr}
  A^{-1} K(z,u)\in B_Z\quad\forall\, (z,u)\in B_Z\times B_U\quad\mbox{and}\quad
  M L_z<1.
\end{equation}
Then for all $u\in B_U$ the equation
\[   
   Az=K(z,u)
\]
has a unique solution $z=z(u)\in B_Z$ and
\begin{equation}\label{zLip}
   \|z(u)-z(\tilde u)\|_Z \le \frac{L_u M}{1-M L_z}
\|u-\tilde u\|_U\quad\forall\, u,\tilde u\in B_U.
\end{equation}
\item[b)] Assume in addition that $B_U$ is
a relatively open convex
subset of $u^*+U_L\subset U$, where $U_L$ is a closed linear subspace of $U$
($U_L=U$ is admitted, then $B_U\subset U$ is convex and open).
and that $K: B_Z\times B_U \to W$ is Fr\'echet differentiable.
Then
$u\in B_U\mapsto z(u)\in Z$ is Fr\'echet differentiable, where
$\delta z_d:= Dz(u) s$ solves for any $d\in U_L$ the problem
\begin{equation}\label{delz}
   A \delta z_d=D_z K(z(u),u) \delta z_d+D_u K(z(u),u) d.
\end{equation}
If $DK: B_Z\times B_U \to {\mathcal L}(Z\times U_L,W)$ is Lipschitz continuous
then also $Dz: B_U\to {\mathcal L}(U_L,Z)$ is Lipschitz continuous.

If $K: B_Z\times B_U \to W$ is k-times Fr\'echet differentiable then
$u\in B_U\mapsto z(u)\in Z$ is k-times Fr\'echet differentiable
and if $D^k K$ is Lipschitz continuous on $B_Z\times B_U$ then
$D^k z$ is Lipschitz continuous on $B_U$.
\end{itemize}
\end{theorem}
\begin{proof}
a): By assumption the mapping $T: (z,u)\in B_Z\times B_U \mapsto A^{-1}
K(z,u)\in B_Z$ is well defined and Lipschitz continuous with
Lipschitz constants $M L_z<1$ with respect to $z$ and $M L_u$
with respect to $u$. Hence, for all $u\in B_U$ there exists a unique
fixed point $z=z(u)$ with $z=T(z,u)$ by Banach's fixed point theorem.

For $u,\tilde u\in B_U$ we obtain
\begin{align*}
  \|z(u)-z(\tilde u)\|_Z &=\|T(z(u),u)-T(z(\tilde u),\tilde u)\|_W\\
  & \le M L_z \|z(u)-z(\tilde u)\|_Z+ M L_u \|u-\tilde u\|_U
\end{align*}
and thus \eqref{zLip}.

b): Now let in addition $B_U$ is relatively open in the closed
affine subspace $u^*+U_L$. Moreover, let
$K: B_Z\times B_U \to W$ be Fr\'echet differentiable and let
$u\in B_U$ be arbitrary. Then $\|D_z K(z(u),u)\|_{{\mathcal L}(Z,W)}\le L_z$
and $\|D_u K(z(u),u)\|_{{\mathcal L}(U_L,W)}\le L_u$ and thus for any
$d\in U_L$ the linear problem \eqref{delz} has by Banach's fixed point theorem
a unique solution $\delta z_d\in Z$.

Since $B_U$ is relatively open in $u^*+U_L$,
we find $\delta>0$ such that $u+d\in B_U$ for all $d\in U_L$ with
$\|d\|_U<\delta$. Then
\begin{align*}
 &A (z(u+d)-z(u)-\delta z_d) = \\
 &=K(z(u+d),u+d)-K(z(u),u)-D_z K(z(u),u) \delta z_d-D_u K(z(u),u) d\\
 &=D_z K(z(u),u) (z(u+d)-z(u)-\delta z_d)+o_W(\|z(u+d)-z(u)\|_Z+\|d\|_U).
\end{align*}
By using \eqref{zLip} we conclude that for $d\in U_L$, $\|d\|_U\to 0$
\[
 \|z(u+d)-z(u)-\delta z_d\|_Z \le \frac{M L_u}{1-M L_z} o_Z(\|d\|_U)=o_Z(\|d\|_U).
\]
If $DK: B_Z\times B_U \to {\mathcal L}(Z\times U_L,W)$ is Lipschitz continuous
then \eqref{delz} can be written as
\[
A \delta z_d=K^{(1)}(\delta z_d,u;d),
\]
where $K^{(1)}(\cdot,\cdot;d): Z\times B_U\to W$ has
for all $d\in U_L$, $\|d\|_U\le 1$ the Lipschitz constant
$L_z$ with respect to $\delta z_d$ and a uniform Lipschitz constant
with respect to $u$. Applying the first part of the theorem again yields
that $Dz: B_U \to {\mathcal L}(U_L,Z)$ is Lipschitz continuous.

Repeating the argument for higher derivatives concludes the proof.
\end{proof}

By applying this result to \eqref{stokesL}--\eqref{Ptrans}, we obtain the following extension of
Theorem \ref{thm:ex}.

\begin{theorem}\label{thm:exdiff}
Let $p>n+3$ and consider any $t_0>0$. Let
$\E(t_0), \F(t_0)$ be defined as in \eqref{Edef} and \eqref{Fdef}
and set with $J=(0,t_0)$
\begin{equation}\label{Udef}
\begin{split}
 &\U_{\hat u}:=W_p^{2-2/p}(\dot \R^{n+1},\R^{n+1}),\quad \U_h:=W_p^{3-2/p}(\R^n),\\
&\U_{\hat c}(t_0):=\F_1(t_0)=L_p(J;L_p(\R^{n+1},\R^{n+1})).
\end{split}
\end{equation}
Then for any $t_0>0$  there exists $\hat\varepsilon_0=\hat\varepsilon_0(t_0)>0$
such that for all data
\[
  (\hat u_0,h_0,\hat c)\in \U_{\hat u}\times \U_h \times \U_{\hat c}(t_0)
\]
satisfying the compatibility condition \eqref{comptrans} (or equivalently
\eqref{comp} with $u_0(x,h_0(x)+y)=\hat u_0(x,y)$) as well as the smallness condition
\begin{equation}\label{initeps}
\|\hat u_0\|_{\U_{\hat u}}+\|h_0\|_{\U_h}+\|\hat c\|_{\U_{\hat c}(t_0)}<
\hat\varepsilon_0
\end{equation}
there exists a unique solution of the transformed problem \eqref{Ptrans} with
\[
 (\hat u,\pi,[\pi],h)\in \E(t_0),
\]
Moreover, the mapping
\[
 \{ (\hat u_0,h_0,\hat c)\in \U_{\hat u}\times \U_h\times \U_{\hat c}(t_0):
 (\hat u_0,h_0,\hat c) \mbox{ satisfy \eqref{comptrans}, \eqref{initeps}}\}
 \mapsto (\hat u,\pi,[\pi],h)\in \E(t_0)
\]
is continuous and infinitely many times differentiable with respect to
$(\hat u_0,\hat c)$. 
\end{theorem}
\begin{proof}
We extend the arguments in \cite{PruessSimonett1} and
apply Theorem \ref{thm:fix} to the transformed formulation \eqref{Ptrans}.

Let $z=(\hat u,\pi,r,h) \in \E(t_0)$ and write \eqref{Ptrans}
\begin{equation}\label{FP1}
  Lz=N(z)+(\hat c,0),\quad (\hat u(0),h(0))=(\hat u_0,h_0),
\end{equation}
where $N$ is defined in \eqref{Ndef}.

Let $(\hat u_0,h_0)$ satisfy \eqref{comptrans} and \eqref{initeps}, where
$\hat\varepsilon_0$ will be adjusted later.

Following \cite{PruessSimonett1}, we first construct $z^*=z^*(\hat u_0,h_0)\in \E(t_0)$ that
satisfies the equation
\begin{equation}\label{zsequ}
  Lz^*=(0,f_d^*,g^*,g_h^*),\quad (\hat u^*(0),h^*(0))=(\hat u_0,h_0),
\end{equation}
where $(0,f_d^*,g^*,g_h^*)\in \F(t_0)$ resolves the compatibility conditions
\eqref{complin1}, \eqref{complin2}.
Then we can write \eqref{FP1} equivalently as
\begin{equation}\label{FP2}
  L \tilde z=N(\tilde z+z^*(\hat u_0,h_0))
 +(\hat c,0)-L z^*(\hat u_0,h_0)=:K(\tilde z;\hat u_0,h_0,\hat c),~~ \tilde z\in {}_0\E(t_0).
\end{equation}
The construction of $z^*$ can be accomplished as in \cite{PruessSimonett1}.
Set
\[
r_0(\hat u_0,h_0)=[\pi_0]:=[\hat \mu \hat \nu(0)^\top {\mathcal D}(\hat
u_0,h_0) \hat \nu(0)]+\sigma (\Delta h_0-G_\kappa(h_0)).
\]
The right hand side consists of several terms of $G(\hat u_0,0,h_0)$ in
\eqref{G0} and thus Lemma \ref{lem:N0} yields that the above mapping
$(\hat u_0,h_0)\in\U_{\hat u}\times \U_h\mapsto [\pi_0]=r_0(\hat u_0,h_0)\in
W^{1-2/p}_p(\R^n)$ is real analytic. Moreover, it is easy to
check that the compatibility conditions hold
\begin{align}\label{compu0}
\begin{aligned}
- [\hat\mu \partial_y v_0]-[\hat\mu\nabla_x w_0] &=G_v(\hat u_0,[\pi_0],h_0)&& \mbox{on $\R^{n}$},\\
- 2[\hat\mu \partial_y w_0]+[\pi_0]-\sigma \Delta h_0 &=G_w(\hat u_0,h_0) && \mbox{on $\R^{n}$},\\
\end{aligned}
\end{align}
Now let $D_{n}=-\Delta$ be the Laplacian in $L_p(\R^{n})$
with domain $H_p^2(\R^{n})$ and set
\[
 g^*(t):= e^{-t D_n} G(\hat u_0,r_0(\hat u_0,h_0),h_0),\quad
 g_h^*(t):=e^{-t D_n} H(\hat u_0,h_0).
\]
By the real analyticity of $r_0(\hat u_0,h_0)$ and Lemma \ref{lem:N0}
the mappings
\begin{align*}
(\hat u_0,h_0)\in \U_{\hat u}\times \U_h &\mapsto G(\hat u_0,r_0(\hat
u_0,h_0),h_0)\in W^{1-2/p}_p(\R^n),\\
(\hat u_0,h_0)\in \U_{\hat u}\times \U_h &\mapsto H(\hat u_0,h_0)\in W^{2-3/p}_p(\R^n)
\end{align*}
are real analytic. Now maximal $L_p$-regularity for $D_n$ yields, see e.g.
\cite[Lem. 8.2]{Escher}
\begin{align*}
  & g^*\in H_p^1(J;W^{-1-1/p}_p(\R^{n}))\cap
L_p(J;W^{1-1/p}_p(\R^{n}))\hookrightarrow \F_3(t_0),\\
  & g_h^*\in H_p^1(J;W^{-1/p}_p(\R^{n}))\cap L_p(J;W^{2-1/p}_p(\R^{n}))\hookrightarrow
\F_4(t_0),
\end{align*}
where the imbeddings follow by real interpolation and $g^*, g_h^*$ are real
analytic in $(\hat u_0,h_0)\in \U_{\hat u}\times\U_h$. \eqref{compu0} ensures
that \eqref{complin2} holds for $g^*$. 
Next, let
\[
  c_d^*(t)=\begin{cases}
  {\mathcal R}_+ e^{-t D_{n+1}} {\mathcal E}_+ v_0^\top \nabla h_0 &
\mbox{in $\R_+^{n+1}$},\\
  {\mathcal R}_- e^{-t D_{n+1}} {\mathcal E}_- v_0^\top \nabla h_0 &
\mbox{in $\R_-^{n+1}$},
  \end{cases}
\]
where ${\mathcal E}_\pm\in {\mathcal L}(W^{2-2/p}(\R_\pm^{n+1}),W^{2-2/p}(\R^{n+1})))$
are extension operators and ${\mathcal R}_\pm$ are the restrictions to
$\R_\pm^{n+1}$. Now $(\hat u_0,h_0)\in \U_{\hat u}\times\U_h \mapsto v_0^\top \nabla h_0\in W^{2-2/p}(\dot\R^{n+1})$
is by Lemma \ref{lem:N0} real analytic.
By $L_p$-regularity for $D_{n+1}$
$c_d^*\in H_p^1(J;L_p(\R^{n+1}))\cap L_p(J;H_p^2(\dot \R^{n+1}))$ and thus
\[
  f_d^*:=\partial_y c_d^*\in \F_2(t_0)\quad\mbox{with}\quad
f_d^*(0)=F_d(v_0,h_0)
\]
is real analytic with respect to $(\hat u_0,h_0)\in \U_{\hat u}\times\U_h$.
Hence, also \eqref{complin1} holds for $f_d^*$ and we conclude that
$R^*:=(0,f_d^*,g^*,g_h^*)\in \F(t_0)$ satisfies the compatibility conditions
\eqref{complin1}, \eqref{complin2} and by construction
$(\hat u_0,h_0)\in \U_{\hat u}\times\U_h\mapsto R^*\in \F(t_0)$ is real analytic. Hence, by Theorem \ref{thm:Lplin}
the linear problem \eqref{zsequ} has a unique solution $z^*=z^*(\hat u_0,h_0)$
that is real analytic and by Lemma \ref{lem:N0} the first derivative
vanishes in $0$, i.e., $Dz^*(0,0)=0$.

Now consider \eqref{FP2}. By construction of $z^*$ the right hand side
of \eqref{FP2} is in ${}_0\F(t_0)$. Denote by
$L_0 \in {\mathcal L}({}_0\E(t_0),{}_0\F(t_0))$ the restriction
of $L$ which is an isomorphism by Corollary \ref{cor:stokes}.
Hence, \eqref{FP2} can be written as
\begin{equation}\label{FP3}
  L_0 \tilde z=N(\tilde z+z^*(\hat u_0,h_0))
 +(\hat c,0)-L z^*(\hat u_0,h_0)=:K(\tilde z;\hat u_0,h_0,\hat c),~~ \tilde z\in {}_0\E(t_0).
\end{equation}
To apply Theorem \ref{thm:fix}
we set now with suitable $\hat\varepsilon_0>0$ and $\delta>0$
\begin{align*}
 B_U(\hat\varepsilon_0)&:= \{ (\hat u_0,h_0,\hat c)\in \U_{\hat u}\times \U_h\times \U_{\hat c}(t_0):\,
 (\hat u_0,h_0,\hat c) \mbox{ satisfy \eqref{comptrans},
\eqref{initeps}}\},\\
 B_Z(\delta)&:=\{ \tilde z\in {}_0\E(t_0):\, \|\tilde
 z\|_{{}_0\E(t_0)}\le\delta\},
\end{align*}
where $\hat\varepsilon_0,\delta>0$ will be adjusted later.

Let $M=\|L_0^{-1}\|_{{\mathcal L}({}_0\F(t_0),{}_0\E(t_0))}$.
We know by Lemma \ref{lem:N} and the properties of $z^*$ that the right hand
side
\begin{equation}\label{Kprop}
(\tilde z,\hat u_0,h_0,\hat c)\in {}_0\E(t_0)\times \U_{\hat u}\times \U_h\times \U_{\hat c}(t_0)
\mapsto K(\tilde z;\hat u_0,h_0,\hat c)\in \F(t_0)
\end{equation}
is real analytic with
\[
 K(0)=0,\quad D_{(\tilde z,\hat u_0,h_0)} K(0)=0.
\]
Hence, the Lipschitz constant $L_z$ of $K$ with respect to $\tilde z$
is arbitrary small close to $0$ and the Lipschitz constant of $K$
with respect to $(\hat u_0,h_0,\hat c)$ is $L_u=2$ close enough to $0$
(note that the Lipschitz constant with respect to $\hat c$ is
$1$). Hence, if we set $\delta=4 M \hat\varepsilon_0$ then for
$\hat\varepsilon_0$ small enough $K$ has the Lipschitz constants
$L_z=1/(2M)$ and $L_u=2$ on $B_Z(\delta)\times B_U(\hat\varepsilon_0)$.
Hence, for all $(\tilde z,\hat u_0,h_0,\hat c)\in B_Z(\delta)\times B_U(\hat\varepsilon_0)$
\begin{align}\label{self}
\begin{split}
& \|L_0^{-1} K(\tilde z;\hat u_0,h_0,\hat c)\|_{{}_0\E(t_0)}
\\&
\le M L_z \|\tilde z\|_{{}_0\E(t_0)}
+ M L_u (\|\hat u_0\|_{\U_{\hat u}}+\|h_0\|_{\U_h}+\|\hat c\|_{\U_{\hat c}(t_0)})
< \frac{1}{2} \delta+2M \hat\varepsilon_0
= \delta. 
\end{split}
\end{align}
Thus, \eqref{contr} is satisfied and
\eqref{FP3} has by Theorem \ref{thm:fix} for all
$(\hat u_0,h_0,\hat c)\in B_U(\hat\varepsilon_0)$ a unique solution
$\tilde z=\tilde z(\hat u_0,h_0,\hat c)\in B_Z(\delta)$ satisfying the
Lipschitz stability \eqref{zLip}.
Since also the real analytic operator $z^*(\hat u_0,h_0)\in \E(t_0)$
is Lipschitz continuous on $B_U(\hat\varepsilon_0)$, the solution
$z(\hat u_0,h_0,\hat c)=\tilde z+z^*\in \E(t_0)$ is unique
and Lipschitz continuous on $B_U(\hat\varepsilon_0)$.

Now let $(\hat u_0^*,h_0^*,\hat c^*)\in B_U(\hat\varepsilon_0)$ be arbitrary.
Then $\{(\hat u_0,h_0^*,\hat c)\in B_U(\hat\varepsilon_0)\}$ is a relatively
open subset of an affine subspace of $\U_{\hat u}\times \U_h\times \U_{\hat c}(t_0)$.
Since \eqref{Kprop} is real analytic, it follows form
Theorem \ref{thm:fix}, b) that $\tilde z(\hat u_0,h_0^*,\hat c)\in {}_0\E(t_0)$
is infinitely many times differentiable with respect to
$(\hat u_0,\hat c)$ and the same holds for
$z(\hat u_0,h_0^*,\hat c)=\tilde z+z^*\in \E(t_0)$.
\end{proof}

\subsection{Results for the original problem}\label{sec:32}
We transfer now the results of Theorem \ref{thm:exdiff} for the transformed problem
\eqref{Ptrans} to the original problem \eqref{P}. To this end, we define for
$f_0\in \U_h$ the spaces
\begin{equation}\label{Udeforg}
\begin{split}
 &\U_{u}(h_0):=W_p^{2-2/p}(\R^{n+1}\setminus\Gamma(0),\R^{n+1}),\quad
\U_{c}(t_0):=L_p(J;H^1_p(\R^{n+1},\R^{n+1})).
\end{split}
\end{equation}
Since the pressure $\pi(t,\cdot)$ is only determined up to a constant, we select from now
on without restriciton the unique representative satisfying (note that $[\pi]$ is
uniquely determined in \eqref{Ptrans})
\begin{equation}\label{prep}
  \pi\in \E_1(t_0),\quad \int_{[-1,1]^n} \pi(t,x,0-)\,dx=0\quad\mbox{for a.a. } t\in (0,t_0).
\end{equation}
Then with the convention \eqref{prep} and by the trace theorem,
we find a Poincar\'e constant $C_P>0$ with
\begin{equation}\label{ppoin}
    \|\pi(t,\cdot)\|_{H_p^1(\dot\R^{n+1})}\le C_P 
    (\|\pi(t,\cdot)\|_{\dot H_p^1(\dot\R^{n+1})}+\|[\pi]\|_{W_p^{1-1/p}(\R^n)})
\end{equation}

The following imbeddings will be useful.
\begin{lemma}
Let $p>n+3$. Then the following imbeddings hold with
$J=(0,t_0)$, $t_0>0$.
\begin{align}\label{imba}
& \E_1(t_0) \hookrightarrow C(\bar J;BUC^1(\dot\R^{n+1},\R^{n+1}))\cap
C(\bar J;BUC(\R^{n+1},\R^{n+1})),\\\label{imbb}
& \E_1(t_0) \hookrightarrow H^1_p(J\times \R^{n+1},\R^{n+1})\cap
C(\bar J;H^1_p(\R^{n+1},\R^{n+1})),\\\label{imbb1}
& \E_1(t_0) \hookrightarrow H^1_p(J\times \R^{n+1},\R^{n+1})\cap
C(\bar J;H^1_{\tilde p}(\R^{n+1},\R^{n+1}))\quad\forall\,\tilde p\in [p,\infty),\\\label{imbc}
& \E_4(t_0) \hookrightarrow C^1(\bar J;BC^1(\R^n))\cap
C(\bar J;BC^2(\R^{n})).
\end{align}
\end{lemma}
\begin{proof}
For the imbeddings \eqref{imba}, \eqref{imbc} see \cite[Lem.
6.1]{PruessSimonett1}.
Moreover, it is obvious that
\[
 \E_1(t_0) \hookrightarrow H^1_p(J\times \dot\R^{n+1},\R^{n+1})
\]
and also $\E_1(t_0) \hookrightarrow C(\bar J;W_p^{2-2/p}(\dot\R^{n+1},\R^{n+1}))$
holds, see \cite[Theorem III.4.10.2]{Amann}
Since the functions $\hat u\in \E_1(t_0)$ are continuous by \eqref{imba} and
thus $[\hat u]=0$, this implies the imbedding \eqref{imbb}.
Now \eqref{imbb1} follows from interpolation between \eqref{imba} and
\eqref{imbb}.
\end{proof}
\begin{theorem}\label{thm:embed}
Let $(\hat u,\pi,[\pi],h)\in\E(t_0)$,
$h_0\in \U_h$, $u_0\in \U_u(h_0)$,
and consider, see \eqref{trafo},
\begin{align}\label{trafoi}
  &u(t,x,y)=\hat u(t,x,y-h(t,x)),~ q(t,x,y)=\pi(t,x,y-h(t,x)),\\
  &u_0(x,y)=\hat u_0(x,y-h_0(x)).\notag
\end{align}
Then there exist constants $C(\|h\|_{\E_4(t_0)})>0$ and $C(\|h_0\|_{\U_h})$ such that
\begin{equation}\label{orgest}
\begin{split}
 &\|u\|_{W_p^1(J\times\R^{n+1},\R^{n+1})}+\|u\|_{L_p(J;H_p^2(\R^{n+1}\setminus\Gamma(t),\R^{n+1})}\le
 C (\|h\|_{\E_4(t_0)}) \|\hat u\|_{\E_1(t_0)},\\\notag
 &\|q\|_{L_p(J;\dot H_p^1(\R^{n+1}\setminus\Gamma(t),\R^{n+1})}
 \le C(\|h\|_{\E_4(t_0)}) \|\pi\|_{\E_2(t_0)},\\
&\|[q]\|_{L_p(J;W_p^{1-1/p}(\Gamma(t)))}
\le 
C (\|h\|_{\E_4(t_0)})
\|[\pi]\|_{L_p(J;W_p^{1-1/p}(\R^n))},\\
& \|\hat u_0\|_{\U_{\hat u}}\le C (\|h_0\|_{\U_h}) \|u_0\|_{\U_{u}(h_0)}.
\end{split}
\end{equation}
\end{theorem}
\begin{proof}
Let $(\hat u,\pi,[\pi],h)\in\E(t_0)$ and consider, see \eqref{trafo},
\[
  u(t,x,y)=\hat u(t,x,y-h(t,x)).
\]
By \eqref{imbc} the mapping $T_{h(t)}: (x,y)\mapsto (x,y-h(t,x))$ is
for all $t\in [0,t_0]$ a $C^2$-diffeomorphism with
$T_{h(t)}^{-1}(x,y)=(x,y+h(t,x))$ and $\mbox{det}(DT_{h(t)}(x,y))=1$.
By \eqref{imbb} the chain rule for Sobolev functions can be applied and
yields $u\in H_p^1(J\times\R^{n+1},\R^{n+1})$ with
\begin{align*}
  \partial_t u(t,x,y)&=\partial_t \hat u(t,T_{h(t)}(x,y))-\partial_y \hat u(t,T_{h(t)}(x,y)) \partial_t h(t,x),\\
  \partial_{(x,y)} u(t,x,y)&=\partial_{(x,y)} \hat u(t,T_{h(t)}(x,y)) DT_{h(t)}(x,y).
\end{align*}
Moreover, again by \eqref{imbc} and $\nabla\hat u\in L_p(J;H^1_p(\dot\R^{n+1},\R^{n+1,n+1})$
we have
\[
 \|u\|_{H_p^1(J\times\R^{n+1},\R^{n+1})}+\|u\|_{L_p(J;H_p^2(\R^{n+1}\setminus\Gamma(t),\R^{n+1})}\le
 C(\|h\|_{\E_4(t_0)}) \|\hat u\|_{\E_1(t_0)}.
\]
Completely analogous one obtains
\[
 \|q\|_{L_p(J;\dot H_p^1(\R^{n+1}\setminus\Gamma(t),\R^{n+1})}
\le C(\|h\|_{\E_4(t_0)}) \|\pi\|_{\E_2(t_0)}.
\]
Now consider $\hat r=[\pi]$ and $r=[q]$ then $r(t,x,h(t,x))=\hat r(t,x)$ and
\begin{align*}
\|r(t,\cdot)\|_{W_p^{1-1/p}(\Gamma(t))}^p&=\int_{\R^n}\int_{\R^n}
\frac{|r(t,x,h(t,x))-r(t,\tilde x,h(t,\tilde x))|^p}{(\sqrt{|x-\tilde x|^2+|h(t,x)-h(t,\tilde
x)|^2})^{n+p-1}}\\
&\quad\qquad\qquad\cdot \sqrt{1+|\nabla h(t,x)|^2} \sqrt{1+|\nabla h(t,\tilde x)|^2}\, dx\,d\tilde x\\
&\le\int_{\R^n}\int_{\R^n}
\frac{|\hat r(t,x)-\hat r(t,\tilde x)|^p}{|x-\tilde x|^{n+p-1}}\, dx\,d\tilde x (1+|h(t,\cdot)|_{BC^1(\R^n)})^2\\
&\le \|\hat r(t,\cdot)\|_{W_p^{1-1/p}(\R^n)}^p C(\|h\|_{\E_4(t_0)})^p.
\end{align*}
Similarly, one obtains also the estimate for $\|\hat u_0\|_{\U_{\hat u}}$,
see \cite[Proof Thm 1.1]{PruessSimonett1}.
\end{proof}
\begin{lemma}\label{lem:uqdiff}
Consider the transformation \eqref{trafoi}, where we choose for 
$\pi\in \E_2(t_0), [\pi]\in \E_3(t_0)$
the unique representative $\pi$ satisfying \eqref{prep}.

Then for all $\tilde p\in [p,\infty)$ the mapping
\begin{equation}\label{udef}
 (\hat u,\pi,[\pi],h)\in\E(t_0)\mapsto 
 u\in C(\bar J;L_{\tilde p}(\R^{n+1},\R^{n+1}))
\end{equation}
is continuously differentiable with derivative
\[
 (\delta\hat u,\delta\pi,[\delta\pi],\delta h)\in\E(t_0)\mapsto \delta u(t,x,y)=
 \delta\hat u(t,x,y-h(t,x))-\partial_y \hat u(t,x,y-h(t,x))\delta h(t,x).
\]
Let ${\mathcal E}_{\pm}\in{\mathcal L}(H^l_p(\R^{n+1}_\pm),H^l_p(\R^{n+1}))$
be extension operators for $l=1,2$ and set
\begin{align}\label{upm}
\begin{split}
\hat u_\pm(t,\cdot)&={\mathcal E}_{\pm} \hat u(t,\cdot),\quad
 u_\pm(t,x,y)=\hat u_\pm(t,T_{h(t)}(x,y)),\\
 \pi_\pm(t,\cdot)&={\mathcal E}_{\pm} \pi(t,\cdot),\quad
 q_\pm(t,x,y)=\pi_\pm(t,T_{h(t)}(x,y)).
\end{split}
\end{align}
Then the mappings
\begin{align}\label{upmdef}
 (\hat u,\pi,[\pi],h)\in\E(t_0)&\mapsto u_\pm \in
L_p(J;H_p^1(\R^{n+1},\R^{n+1})),\\
\label{qpmdef}
 (\hat u,\pi,[\pi],h)\in\E(t_0)&\mapsto q_\pm\in L_p(J;L_p(\R^{n+1}))
\end{align}
are continuously differentiable with derivative
\begin{align*}
 (\delta\hat u,\delta\pi,[\delta\pi],\delta h)\in\E(t_0)&\mapsto 
 \binom{\delta u_\pm}{\delta q_\pm}(t,x,y)\\
 & =\binom{\delta\hat u_\pm}{\delta\pi_\pm}(t,x,y-h(t,x))
 -\partial_y \binom{\hat u_\pm}{\pi_\pm}(t,x,y-h(t,x)) \delta h(t,x).
\end{align*}
\end{lemma}
\begin{proof}
Define as in the previous proof the $C^2$-diffeomorphisms $T_{h(t)}: (x,y)\mapsto
(x,y-h(t,x))$. Then $u(t,x,y)=\hat u(t,T_{h(t)}(x,y))$. Let
$(\hat u,\pi,[\pi],h),(\delta \hat u,\delta\pi,[\delta\pi],\delta h)\in\E(t_0)$
be arbitrary. We recall the well known fact that for any $v\in
C(\bar J;L_{\tilde p}(\R^{n+1}))$, $p\le\tilde p<\infty$, it holds
\begin{equation}\label{Cshift}
 \sup_{t\in J} \|v(t,T_{(h+\delta h)(t)}(\cdot))-v(t,T_{h(t)}(\cdot))\|_{L_{\tilde p}(\R^{n+1})}
 \to 0\quad \mbox{as } \|\delta h\|_{C(\bar J;BC^1(\R^n))}\to 0,
\end{equation}
which can be shown by an approximation of $v$ through a sequence of
continuous functions with compact support. Similarly, for
$v\in L_p(J;L_p(\R^{n+1}))$ one has
\begin{equation}\label{Lpshift}
 \int_J \|v(t,T_{(h+\delta h)(t)}(\cdot))-v(t,T_{h(t)}(\cdot))\|_{L_p(\R^{n+1})}^p\,dt
 \to 0\quad \mbox{as } \|\delta h\|_{C(\bar J;BC^1(\R^n))}\to 0.
\end{equation}
Consider the remainder term
\begin{align}\label{Ru}
\begin{split}
R_u(t,x,y):=&(\hat u+\delta \hat u)(t,T_{(h+\delta h)(t)}(x,y))-\hat
u(t,T_{h(t)}(x,y))\\
&-\delta \hat u(t,T_{h(t)}(x,y))+\partial_y \hat u(t,T_{h(t)}(x,y)) \delta
h(t,x).
\end{split}
\end{align}
Let $p\le\tilde p<\infty$ be arbitrary. We obtain
\begin{align*}
&\|R_u\|_{C(\bar J;L_{\tilde p}(\R^{n+1}))}
\le \sup_{t\in J}\left\|\int_0^1 (\partial_y \delta\hat u(t,T_{(h+\tau\delta
h)(t)}(\cdot))\,d\tau\, \delta h(t,\cdot)\right\|_{L_{\tilde p}(\R^{n+1})}\\
&+\sup_{t\in J}\left\|\int_0^1 (\partial_y \hat u(t,T_{(h+\tau\delta h)(t)}(\cdot))
-\partial_y \hat u(t,T_{h(t)}(\cdot)))\,d\tau\, \delta
h(t,\cdot)\right\|_{L_{\tilde p}(\R^{n+1})}
\end{align*}
\begin{align*}
&\le \|\delta\hat u\|_{C(\bar J;H^1_{\tilde p}(\R^{n+1})} \|\delta h\|_{C(\bar J;BC(\R^n))}\\
&+\sup_{t\in J,\tau\in [0,1]}
\|\partial_y \hat u(t,T_{(h+\tau\delta h)(t)}(\cdot))
-\partial_y \hat u(t,T_{h(t)}(\cdot))\|_{L_{\tilde p}(\R^{n+1})}\|\delta
h\|_{C(\bar J;BC(\R^n))}\\
&=\|\delta\hat u\|_{C(\bar J;H^1_{\tilde p}(\R^{n+1})} \|\delta h\|_{C(\bar J;BC(\R^n))}+
o(\|\delta h\|_{C(\bar J;BC(\R^n))})\\
&=o(\|\delta\hat u\|_{\E_1(t_0)}+\|\delta h\|_{\E_4(t_0)}).
\end{align*}
Here, we have used \eqref{Cshift} and the imbeddings \eqref{imbb1},
\eqref{imbc}. This shows that \eqref{udef} is Fr\'echet differentiable. The
continuity of the derivative follows from the fact that for
$(\hat u_1,\pi_1,[\pi_1],h_1)\to (\hat u,\pi,[\pi],h)$ in $\E(t_0)$ we have
\begin{align*}
&\sup_{t\in J} \|\delta \hat u(t,T_{h_1(t)}(\cdot))
-\delta \hat u(t,T_{h(t)}(\cdot))\|_{L_{\tilde p}(\R^{n+1})}\\
&=\sup_{t\in J}\left\|\int_0^1 \partial_y \delta \hat u(t,T_{(h+\tau (h_1-h))(t)}(\cdot))\,d\tau\,
(h_1-h)(t,\cdot)\right\|_{L_{\tilde p}(\R^{n+1})}\\
&\le \|\delta\hat u\|_{C(\bar J;H^1_{\tilde p}(\R^{n+1}))} \|\tilde h-h\|_{C(\bar J;BC(\R^n))}\\
&\le C \|\delta\hat u\|_{\E_1(t_0)} \|\tilde h-h\|_{\E_4(t_0)}
\end{align*}
as well as $\|\delta h\|_{C(\bar J;BC(\R^n))}\le C \|\delta h\|_{\E_4(t_0)}$ and
\begin{align*}
&\sup_{t\in J} \|\partial_y \hat u_1(t,T_{h_1(t)}(\cdot))
-\partial_y \hat u(t,T_{h(t)}(\cdot))\|_{L_{\tilde p}(\R^{n+1})}\\
&\le \|\hat u_1-\hat u\|_{C(\bar J;H_{\tilde p}^1(\R^{n+1}))}+
\sup_{t\in J} \|(\partial_y \hat u_1(t,T_{h_1(t)}(\cdot))
-\partial_y \hat u(t,T_{h(t)}(\cdot)))\|_{L_{\tilde p}(\R^{n+1})}\to 0,
\end{align*}
where we have used \eqref{Cshift}.

The continuous differentiability of \eqref{qpmdef} follows
very similarly by using \eqref{Lpshift} instead of \eqref{Cshift}
and by applying \eqref{prep}, \eqref{ppoin} and Theorem \ref{thm:embed}.

Finally, consider \eqref{upmdef}, \eqref{upm}.
Then $\hat u_\pm, \delta \hat u_\pm\in L_p(J;H^2_p(\R^{n+1},\R^{n+1}))$.
Define the remainder terms $R_{u_\pm}$ as in \eqref{Ru} with $\hat u,\hat \delta u$
replaced by $\hat u_\pm,\delta \hat u_\pm$. After differentiation
a calculation as above yields
\begin{align*}
&\|\nabla R_{u_\pm}\|_{L_p(J;L_p(\R^{n+1}))}\\
&\le \left\|\int_0^1 (\partial_y \nabla\delta\hat u_\pm(t,T_{(h+\tau\delta
h)(t)}(\cdot))^\top DT_{(h+\tau\delta h)(t)}\,d\tau\, \delta
h(t,\cdot)\right\|_{L_p(J;L_p(\R^{n+1}))}\\
&+\left\|\int_0^1 \partial_y \delta\hat u_\pm(t,T_{(h+\tau\delta
h)(t)}(\cdot))\,d\tau\, \nabla\delta
h(t,\cdot)^\top\right\|_{L_p(J;L_p(\R^{n+1}))}\\
&+\left\|\int_0^1 \nabla (\partial_y \hat u_\pm(t,T_{(h+\tau\delta
h)(t)}(\cdot)) 
-\partial_y \hat u_\pm(t,T_{h(t)}(\cdot)))^\top 
DT_{h(t)}\,d\tau\, \delta h(t,\cdot)\right\|_{L_p(J;L_p(\R^{n+1}))}\\
&+\left\|\int_0^1 \partial_y \nabla\hat u_\pm(t,T_{(h+\tau\delta h)(t)}(\cdot))^\top 
(DT_{(h+\tau\delta h)(t)}-DT_{h(t)})\,d\tau\, \delta
h(t,\cdot)\right\|_{L_p(J;L_p(\R^{n+1}))}\\
&+\left\|\int_0^1 (\partial_y \hat u_\pm(t,T_{(h+\tau\delta h)(t)}(\cdot))
-\partial_y \hat u_\pm(t,T_{h(t)}(\cdot)))\,d\tau\, \nabla\delta
h(t,\cdot)^\top\right\|_{L_p(J;L_p(\R^{n+1}))}\\
&\le \|\delta\hat u_\pm\|_{L_p(J;H^2_p(\R^{n+1})} 
(1+\|h\|_{C(\bar J;BC^1(\R^n))}+\|\delta h\|_{C(\bar J;BC^1(\R^n))})
\|\delta h\|_{C(\bar J;BC(\R^n))}\\
&+\|\delta\hat u_\pm\|_{L_p(J;H^1_p(\R^{n+1})}\|\delta
h\|_{C(\bar J;BC^1(\R^n))}\\
&+\|\partial_y \nabla\hat u_\pm(t,T_{(h+\tau\delta h)(t)}(\cdot))
-\partial_y \nabla\hat u_\pm(t,T_{h(t)}(\cdot))\|_{L_p(J;L_p(\R^{n+1}))}\\
&\cdot(1+\|h\|_{C(\bar J;BC^1(\R^n))}) \|\delta h\|_{C(\bar J;BC(\R^n))}
+\|\hat u_\pm\|_{L_p(J;H_p^2(\R^{n+1}))}\|\delta h\|_{C(\bar J;BC(\R^n))}^2\\
&+\|\partial_y \hat u_\pm(t,T_{(h+\tau\delta h)(t)}(\cdot))
-\partial_y \hat u_\pm(t,T_{h(t)}(\cdot))\|_{L_p(J;L_p(\R^{n+1}))}
\|\delta h\|_{C(\bar J;BC^1(\R^n))}\\
&=o(\|\delta\hat u\|_{L_p(J;H^2_p(\R^{n+1})}+\|\delta h\|_{C(\bar J;BC^1(\R^n))})
=o(\|\delta\hat u\|_{\E_1(t_0)}+\|\delta h\|_{\E_4(t_0)}).
\end{align*}
Here we have used \eqref{Lpshift} and the imbedding \eqref{imbc}.
The continuity of the derivative follows with very similar estimates.
\end{proof}
Similarly, we have
\begin{lemma}\label{lem:cdiff}
Let $\U_c(t_0)=L_p(J;H^1_p(\R^{n+1}))$. Then the mapping
\begin{equation}\label{hatc}
  (c,h)\in \U_c(t_0)\times \E_4(t_0)\mapsto \hat c(c,h) \in \U_{\hat c}(t_0)
\end{equation}
with $\hat c(c,h)(t,x,y)=c(t,x,y+h(t,x))$ is continuously differentiable
with derivative
\[
 (\delta c,\delta h)\in \U_c(t_0)\times \E_4(t_0)\mapsto
 \delta c(t,x,y+h(t,x))+\partial_y c(t,x,y+h(t,x)) \delta h(t,x).
\]
\end{lemma}
\begin{proof}
The proof is the same as for \eqref{qpmdef}.
\end{proof}
For the original data $(u_0,h_0,c)$ we obtain the following
existence and differentiability result.
\begin{theorem}\label{thm:difforg}
Let $p>n+3$ and $\U_{u}(h_0), \U_{c}(t_0)$ be defined by
\eqref{Udeforg}.
Then for any $t_0>0$  there exists $\varepsilon_0=\varepsilon_0(t_0)>0$
such that for all data
\[
  (h_0,c)\in \U_h \times \U_c(t_0),\quad u_0\in \U_{u}(h_0)
\]
satisfying the compatibility condition \eqref{comp} as well as the smallness condition
\begin{equation}\label{initepsorg}
\|u_0\|_{\U_{u}(h_0)}+\|h_0\|_{\U_h}+\|c\|_{\U_c(t_0)}<
\varepsilon_0
\end{equation}
there exists a unique solution of the transformed problem \eqref{Ptrans} with
\[
 (\hat u,\pi,[\pi],h)\in \E(t_0),
\]
Moreover, for any $h_0$ with $\|h_0\|_{\U_h}<\varepsilon_0$ the mapping
\[
 \{ (u_0,c)\in \U_u(h_0)\times \U_{c}(t_0):
 (u_0,h_0,c) \mbox{ satisfy \eqref{comp}, \eqref{initepsorg}}\}
 \mapsto (\hat u,\pi,[\pi],h)\in \E(t_0)
\]
is continuously differentiable.

By the chain rule, also the original state $(u,q)$ depends continuously differentiable
on $(u_0,c)$ with the spaces given in \eqref{udef}, \eqref{upmdef},
\eqref{qpmdef}.
\end{theorem}
\begin{proof}
We adapt the fixed point argument in the proof of Theorem \ref{thm:exdiff}.
Let
\begin{equation}\label{hatcdef}
\hat c(c,h)(t,x,y)=c(t,x,y+h(t,x)).
\end{equation}
The only difference compared to
the situation in Theorem \ref{thm:exdiff} results from the fact that
$\hat c(c,h)$ depends now on $h$. Hence, the fixed point equation
\eqref{FP3} changes to
\begin{equation}\label{FP3mod}
  L_0 \tilde z=K(\tilde z;\hat u_0,h_0,\hat c(c,\tilde z+z^*(\hat u_0,h_0))),~~ \tilde z\in {}_0\E(t_0).
\end{equation}

Let $\hat\varepsilon_0>0$ be as in Theorem \ref{thm:exdiff}.
We have
\begin{equation}\label{cnorm}
\|\hat c(c,h)\|_{\U_{\hat u}}=\|c\|_{\U_{\hat u}}
\end{equation}
and the last estimate in
\eqref{orgest} shows that for $\varepsilon_0>0$ small enough
\eqref{initepsorg} implies \eqref{initeps}. 

Hence, for all $(u_0,h_0,c)$ satisfying \eqref{initepsorg}
we have $(\hat u_0,h_0,\hat c(c,h))\in B_U(\hat\varepsilon_0)$
(note that \eqref{cnorm} holds independently of $h$) and thus by
\eqref{self}
\begin{align*}
& \|L_0^{-1} K(\tilde z;\hat u_0,h_0,\hat c)\|_{{}_0\E(t_0)}<\delta.
\end{align*}
Finally, the Lipschitz constant of $K(\tilde z;\hat u_0,h_0,\hat c)$
with respect to $\hat c$ is $1$ and the mapping \eqref{hatc}, \eqref{hatcdef}
is by Lemma \ref{lem:cdiff} continuously differentiable and the Lipschitz constant
with respect to $h$ is bounded by $\|c\|_{\U_c(t_0)}<\varepsilon_0$. Hence, for
$\varepsilon_0>0$ small enough, \eqref{FP3mod} is a contraction
and the existence, uniqueness and continuous differentiability follow as
in the proof of Theorem \ref{thm:exdiff}.

Lemma \ref{lem:uqdiff} and the chain rule yield now the continuous
differentiability of the original state $(u,q)$ with respect to
$(u_0,c)$ for the spaces given in \eqref{udef}, \eqref{upmdef},
\eqref{qpmdef}.
\end{proof}

\subsection{Volume-of-Fluid type formulation}\label{sec:33}
Our aim is finally to derive a Volume-of-Fluid (VoF) type formulation with
corresponding sensitivity equation that is satisfied by the solution
$(u,q)$ of the problem \eqref{P} and its sensitivities $(\delta u,\delta q)$.
This provides an analytical foundation to derive and analyze appropriate
numerical VoF schemes for sensitivity calculations.

Let $\alpha:\R^{n+1}\to [0,1]$ be a phase indicator satisfying the
transport equation
\begin{align}
\label{transp}
\partial_t \alpha+u\cdot\nabla\alpha=0&\quad \mbox{in $J\times
\R^{n+1}$,}\quad
\alpha(0)=1_{\Omega_1(0)}\quad \mbox{on $\R^{n+1}$}.
\end{align}
We note that for $u\in L_1(J;W^1_\infty(\R^{n+1};\R^{n+1}))$ with
$\Div u=0$ a.e. any distributional solution $\alpha\in L_1(J;L_{1,loc}(\R^{n+1}))$
is also a distributional solution of
\begin{equation}
\label{transpc}
\partial_t \alpha+\Div(u \alpha)=0\quad \mbox{in $J\times \R^{n+1}$},
\quad
\alpha(0)=1_{\Omega_1(0)}\quad \mbox{on $\R^{n+1}$}.
\end{equation}
We define now
\[
 \rho(\alpha)=\alpha\rho_1+(1-\alpha)\rho_2,\quad
 \mu(\alpha)=\alpha\mu_1+(1-\alpha)\mu_2.
\]
We will show that the unique solution $(u,q)$ of \eqref{P} according to Theorem \ref{thm:difforg}
satisfies the VoF-type formulation
\begin{align}\notag
&\int_{\R^{n+1}} \bigl(\partial_t (\rho(\alpha) u)+\Div(\rho(\alpha) u\otimes
u))(t,x,y)^\top \varphi(x,y)\\
&\qquad\qquad+S(u,q;\mu(\alpha))(t,x,y):\nabla\varphi(x,y)\bigr)\,d(x,y)
\qquad\forall\,\varphi\in C_c^1(\R^{n+1};\R^{n+1})\label{weakVoF1}\\\notag
&=-\lim_{\varepsilon\searrow 0}\int_{\R^{n+1}} \sigma
\frac{\nu_{\varepsilon}(t,x,y)^\top}{|\nu_{\varepsilon}(t,x,y)|}
 (D\varphi-\Div(\varphi)I)(x,y)
\nabla\alpha(t,x,y)\,d(x,y),\\\label{weakVoF2}
&\int_{\R^n} \Div(u)\, \psi\,dx=0\quad\forall\,\psi\in C_c^1(\R^{n+1}),\\
&\alpha \mbox{ satisfies \eqref{transpc}},\label{weakVoF3}
\end{align}
where $\nu_{\varepsilon}$ is a suitable smoothed normal computed from
$\nabla\alpha$, see \eqref{nueps} below.

In order to deal with the sensitivity equation, it will be beneficial to
consider measure-valued solutions of the general equation
\begin{equation}
\label{transpcmu}
\partial_t \delta\alpha+\Div(u \,\delta\alpha)=b\quad \mbox{in $J\times \R^{n+1}$},\quad
 \delta\alpha(0)=\delta\alpha_0\quad \mbox{on $\R^{n+1}$}.
\end{equation}
For $u\in L_1(J;W^1_\infty(\R^{n+1};\R^{n+1}))$ we can define uniquely
the continuous mapping $(x,y)\mapsto X(t;x,y)$, where $X(t;x,y)$ satisfies the
characteristic equation
\begin{equation}\label{char}
 \partial_t X(t;x,y)=
 u(t,X(t;x,y)),~~ t\in J,\quad X(0;x,y)=(x,y).
\end{equation}
In the following, we denote by ${\mathcal M}_{loc}(\R^{n+1})$ the space of
locally bounded Radon measures.
\begin{proposition}\label{prop:MV}
Let $u\in L_1(J;W^1_\infty(\R^{n+1};\R^{n+1}))$. Then for any
$\delta\alpha_0\in {\mathcal M}_{loc}(\R^{n+1})$ there exists a unique
distributional solution of \eqref{transpcmu} in
$C(\bar J;{\mathcal M}_{loc}(\R^{n+1})-\mbox{weak$^*$})$, given by
\begin{equation}\label{transpi1}
 \delta\alpha(t)=X(t)(\delta\alpha_0)+\int_0^t X(t-s)(b(s)).
\end{equation}
Here, $X$ is the forward flow defined by \eqref{char} and
$\delta\alpha_t=X(t)(\delta\alpha_0)$ is the measure satisfying
\[
  \int_{\R^{n+1}} \phi(x,y) \,d\delta \alpha_t(x,y)=
  \int_{\R^{n+1}} \phi(X(t;x,y))\,d\delta\alpha_0(x,y)\quad\forall\,\phi\in
C_c(\R^{n+1}).
\]
\end{proposition}
\begin{proof}
For $u\in L_1(J;C^1(\R^{n+1};\R^{n+1}))$, see \cite[Thm. 3.1 and
3.3]{PoupaudRascle}. Since the characteristics are unique and stable also
for $u\in L_1(J;W^1_\infty(\R^{n+1};\R^{n+1}))$, the proofs directly extend to this
case, see also \cite{AmbrosioBook}. 
\end{proof}

\begin{proposition}\label{prop:alpha}
If $\hat u\in \E_1(t_0)$, $[\hat u]=0$ and $u$ is given by \eqref{trafoi}
then \eqref{transp} as well as \eqref{transpc} have a unique solution given by
\begin{equation}\label{alsol}
   \alpha(t,X(t;x,y))=1_{\Omega_1(0)}(x,y)
\end{equation}
and thus $\alpha(t,\cdot)=1_{\Omega_1(t)}$.

Moreover, for $\varepsilon_0$ from Theorem \ref{thm:difforg}
and any $h_0$ with $\|h_0\|_{\U_h}<\varepsilon_0$ the mapping
\[
 \{ (u_0,c)\in \U_u(h_0)\times \U_{c}(t_0):
 (u_0,h_0,c) \mbox{ satisfy \eqref{comp}, \eqref{initepsorg}}\}
 \mapsto \alpha\in C(\bar J;{\mathcal M}_{loc}(\R^{n+1})-\mbox{weak$^*$})
\]
is continuously differentiable. The derivative
\[
 (\delta u_0,\delta c)\in \U_u(h_0)\times \U_{c}(t_0)
 \mapsto  \delta\alpha \in C(\bar J;{\mathcal M}_{loc}(\R^{n+1})-\mbox{weak$^*$})
\]
is given by the unique measure-valued solution of
\begin{equation}
\label{transpclin}
\partial_t \delta\alpha+\Div(u\, \delta\alpha)=-\Div(\delta u\, \alpha)\quad \mbox{in $J\times \R^{n+1}$},
\quad
\delta\alpha(0)=0\quad \mbox{on $\R^{n+1}$}.
\end{equation}
Finally, $\delta\alpha$ satisfies
\begin{equation}\label{delalpha}
  \int_{\R^{n+1}} \phi(x,y)\,d\delta\alpha(t)(x,y)=\int_{\R^n} \phi(x,h(t,x)) \delta h(t,x)\,dx.
\end{equation}
\end{proposition}
\begin{proof}
If $\hat u\in \E_1(t_0)$, $[\hat u]=0$ and $u$ is given by \eqref{trafoi}
then $u\in C(\bar J;W^1_\infty(\R^{n+1};\R^{n+1}))$ by \eqref{imba},
\eqref{imbc}. Now it is well known that
\eqref{alsol} provides the unique weak solution of \eqref{transp} in
$L_{1,loc}(J\times\R^{n+1})$, see \cite[Prop. 2.2]{AmbrosioBook}
and \cite[Cor. II.1]{DiPernaLions}.
Since $\Div(u)=0$ a.e., it is also a distributional solution of
\eqref{transpc}, which is unique by Proposition \ref{prop:MV}.

Let now $(u_0,h_0,c), (\delta u_0,0,\delta c)\in \U_u\times \U_h\times \U_c(h_0)$
be such that $(u_0,h_0,c)$ and $(u_0,h_0,c)+(\delta u_0,0,\delta c)$
satisfy the conditions of Theorem \ref{thm:difforg}.
Denote by $(\hat u,\pi,[\pi],h)$ the unique solution of
\eqref{Ptrans} for data $(u_0,h_0,c)$ and by
$(\hat u^s,\pi^s,[\pi^s],h^s)$ the one
for data $(u_0,h_0,c)+s\,(\delta u_0,0,\delta c)$. Let
$(u,q)$ and $(u^s,q^s)$ be the corresponding states
in physical coordinates according to \eqref{trafo} and let
$\alpha=1_{\Omega_1(t)}, \alpha^s=1_{\Omega_1^s(t)}$ be the corresponding
solutions of \eqref{transpc}. Finally, let
$(\delta u,\delta h,\delta q)$ be the directional derivatives
(sensitivities) in direction $(\delta u_0,0,\delta c)$ which exist by
Theorem \ref{thm:difforg}. We show that
\begin{equation}\label{eq:dal}
 \frac{\alpha^s-\alpha}{s} \to \delta\alpha
\quad\mbox{ in $C(\bar J;{\mathcal M}_{loc}(\R^{n+1})-\mbox{weak$^*$})$ as $s\to 0$},
\end{equation}
where $\delta\alpha$ solves \eqref{transpclin}.
Let $\phi\in C_c(\R^{n+1})$ be arbitrary. Then
\begin{align*}
 \int_{\R^{n+1}} \frac{\alpha^s-\alpha}{s}(t,x,y)\phi(x,y)\,d(x,y)&=
 \int_{\R^n} \int_{h(t,x)}^{h^s(t,x)} \frac{1}{s} \phi(x,y)\,d(x,y)\\
 &\to \int_{\R^n} \phi(x,h(t,x)) \delta h(t,x)\,dx
\end{align*}
as $s\to 0$ uniformly in $t\in \bar J$, where we have used the
differentiability result of Theorem \ref{thm:difforg}. Moreover, it is
obvious that the middle term is continuous with respect to $t$. Hence,
\eqref{eq:dal} is proven and we have only to show that
$\delta\alpha$ solves \eqref{transpclin}.

To this end, let $\varphi\in C_c^1(J\times\R^{n+1})$ be arbitrary. Since
$\alpha,\alpha^s$ are distributional solutions of \eqref{transpc}, we have
\begin{align*}
0&=\int_J \int_{\R^{n+1}} -\left(\partial_t \varphi+(u\cdot\nabla) \varphi)
\tfrac{\alpha^s-\alpha}{s}+\alpha^s
(\tfrac{u^s-u}{s}\cdot\nabla)\varphi\right)(t,x,y)\,d(x,y)\,dt\\
&\to \int_J \int_{\R^{n+1}} -\left(\partial_t \varphi+(u\cdot\nabla) \varphi)
\delta\alpha+\alpha
(\delta u\cdot\nabla)\varphi\right)(t,x,y)\,d(x,y)\,dt
\end{align*}
as $s\to 0$. For the limit transition, we have used $u\in C(\bar
J;W^1_\infty(\R^{n+1})$, \eqref{eq:dal} and that by Theorem \ref{thm:difforg}
$\alpha^s=1_{\Omega^s(t)}\to\alpha=1_{\Omega(t)}$
in $L_{2,loc}(J\times \R^{n+1})$ and
$\frac{u^s-u}{s}\to\delta u$ in $C(\bar J;L_p(\R^{n+1}))$.
Hence, $\delta\alpha$ is a distributional solution of
\eqref{transpclin}, which is unique by Proposition \ref{prop:MV}.
\end{proof}

The next step is to express the surface tension term by using the phase
indicator $\alpha$ such that its sensitivities can be expressed by using the
measure $\delta\alpha$.

We first rewrite the surface tension term in the weak formulation
\eqref{weak1}.
\begin{lemma}
Let $\varphi\in C_c^1(\R^{n+1})$. Then with the curvature $\kappa(t)$ of
$\Gamma(t)$ according to \eqref{nukappa} one has the identity
\begin{align}\notag
&\int_{\Gamma(t)} (\sigma \kappa \nu)(t,x,y)^\top
\varphi(x,y)\,dS(x,y)\\\notag
&=\int_{\R^n} \sigma \Div\nolimits_x\left(\frac{\nabla h(t,x)}{\sqrt{1+|\nabla h(t,x)|^2}}\right)
\binom{-\nabla h(t,x)}{1}^\top \varphi(x,h(t,x))\,dx\\\label{surface}
&=\int_{\R^n}\sigma\frac{(\nabla h(t,x)^\top,-1)}{\sqrt{1+|\nabla h(t,x)|^2}}
(D\varphi(x,h(t,x))-\Div(\varphi)(x,h(t,x))I) 
\binom{\nabla h(t,x)}{-1}\,dx.
\end{align}
\end{lemma}
\begin{proof}
The first identity follows directly from \eqref{nukappa}.
The second one follows from integration by parts and
reflects the well known identity from
differential geometry, see for example \cite[Lem. 2.1]{Deckelnick}
\[
\int_{\Gamma(t)} (\kappa \nu)(t,x,y)^\top \varphi(x,y)\,dS(x,y)=
-\int_{\Gamma(t)} \nabla_T \mathop{\text{id}}\nolimits_{\Gamma(t)}(x,y) : \nabla_T
\varphi(x,y)\,dS(x,y),
\]
where $\nabla_T \varphi_i=\nabla \varphi_i-\nu^\top \nabla \varphi_i \nu$ is
the tangential derivative.
\end{proof}
To compute the interface normal from $\nabla\alpha$, we use the following
simple fact.
\begin{lemma}\label{lem:normal}
Let $\psi\in C_c^1(\R^{n+1};\R^{n+1})$. Then
\begin{align*}
 -\int_{\R^{n+1}} \psi(x,y)^\top \nabla\alpha(t,x,y)\,d(x,y)&=
 \int_{\Gamma(t)} \psi(x,y)^\top\nu(t,x,y)\,dS(x,y)\\
 &=
 \int_{\R^n} \psi(x,h(t,x))^\top \binom{-\nabla h(t,x)}{1}\,dx.
\end{align*}
\end{lemma}
\begin{proof}
By the definition of distributional derivatives one has
\begin{align*}
&-\int_{\R^{n+1}} \psi(x,y)^\top \nabla\alpha(t,x,y)\,d(x,y)=
\int_{\R^{n+1}} \Div(\psi)(x,y)\alpha(t,x,y)\,d(x,y)\\
&=
\int_{\Omega_1(t)} \Div(\psi)(x,y)\,d(x,y)=
\int_{\Gamma(t)} \psi(x,y)^\top\nu(t,x,y)\,dS(x,y)\\
&=\int_{\R^n} \psi(x,h(t,x))^\top \binom{-\nabla h(t,x)}{1}\,dx.
\end{align*}
\end{proof}
Let now $\delta \in (0,1/2)$ and
\[
\psi_\delta\in C_c^1((-1,1)),\quad
\psi_\delta|_{[-1+\delta,1-\delta]}\equiv 1,\quad
\psi_\delta(-s)=\psi_\delta(s)\quad\forall\,s\in\R,
\int_\R \psi_\delta(s)\,ds=1.
\]
and set
\[
 \phi_{\varepsilon}(x,y)=\frac{1}{\varepsilon^{n}}
\psi_\delta(y/\varepsilon)\prod_{i=1}^n\psi_\delta(x_i/\varepsilon).
\]
To recover a mollified normal (not necessarily of unit length) we use
\begin{align}\label{nueps1}
\nu_{\varepsilon}(t,x,y)&:=
-\int_{\R^{n+1}} \phi_{\varepsilon}((\tilde x,\tilde y)-(x,y))
\nabla\alpha(t,\tilde x,\tilde y)\,d(\tilde x,\tilde y).
\end{align}
Then by Lemma \ref{lem:normal}
\begin{align*}
\nu_{\varepsilon}(t,x,y)
&=\int_{\Gamma(t)} \phi_{\varepsilon}((\tilde x,\tilde y)-(x,y))
\nu(t,\tilde x,\tilde y)\,dS(\tilde x,\tilde y)\\
&=\int_{\R^n} \phi_{\varepsilon}((\tilde x,h(t,\tilde
x))-(x,y))\binom{-\nabla h(t,\tilde x)}{1}\,d\tilde x.
\end{align*}
Now assume that
\begin{equation}\label{Dhbnd}
|\nabla h|\le 1-\delta\quad\mbox{on }
x+[-\varepsilon,\varepsilon]^n.
\end{equation}
Then we have by the definition of $\phi_{\varepsilon}$
\begin{align}\label{nueps}
\nu_{\varepsilon}(t,x,h(t,x))&=
\frac{1}{\varepsilon^n}\int_{\R^n} \prod_{i=1}^n\psi_\delta((\tilde x_i-x_i)/\varepsilon)
\binom{-\nabla h(t,\tilde x)}{1}\,d\tilde x.
\end{align}
\begin{lemma}\label{lem:nueps}
Let \eqref{Dhbnd} hold. If $h\in C(\bar J;BC^2(\R^n))$ then there is $C>0$ such that
\[
  |\nu_{\varepsilon}(t,x,h(t,x))-(-\nabla h(t,x),1)^\top|
\le C \varepsilon\quad\forall\,
(t,x)\in J\times \R^n.
\]
On compact subsets the error is $o(\varepsilon)$.
\end{lemma}
\begin{proof}
Since $\nabla h$ has a uniform Lipschitz constant with respect to $x$ the
first assertion follows immediately from \eqref{nueps}.
Moreover, since $\nabla h(t,\tilde x)=\nabla h(t,x)+\nabla^2 h(t,x) (\tilde x-x)
+o(\|\tilde x-x\|)$, the $o(\varepsilon)$ is obtained by the symmetry of
$\psi_\delta$.
\end{proof}
The variation of $\nu_{\varepsilon}$ is
\begin{align}\label{delnueps}
\delta\nu_{\varepsilon}(t,x,y)&:=
-\int_{\R^{n+1}} \phi_{\varepsilon}((\tilde x,\tilde y)-(x,y))
\nabla d\delta\alpha(t)(\tilde x,\tilde y).
\end{align}
with the measure-valued solution of \eqref{transpclin}.
\begin{lemma}\label{lem:delnueps}
Let \eqref{Dhbnd} hold. If $\delta h\in C(\bar J;BC^2(\R^n))$ then there is $C>0$ such that
\[
  |\delta\nu_{\varepsilon}(t,x,h(t,x))-(-\nabla \delta h(t,x),1)^\top|
\le C \varepsilon\quad\forall\,
(t,x)\in J\times \R^n.
\]
On compact subsets the error is $o(\varepsilon)$.
\end{lemma}
\begin{proof}
Then by \eqref{delalpha}
\begin{align*}
\delta\nu_{\varepsilon}(t,x,y)&:=
\int_{\R^{n+1}} \nabla\phi_{\varepsilon}((\tilde x,\tilde y)-(x,y))
d\delta\alpha(t)(\tilde x,\tilde y)\\
&=\int_{\R^n} \nabla\phi_{\varepsilon}((\tilde x,h(t,\tilde x))-(x,y))
\delta h(t,\tilde x)\,d\tilde x.
\end{align*}
Setting $y=h(t,x)$ and using \eqref{Dhbnd} we obtain
\begin{align*}
\delta\nu_{\varepsilon}(t,x,h(t,x))&=
\frac{1}{\varepsilon^n}\int_{\R^n} 
\binom{\nabla_{\tilde x} \prod_{i=1}^n\psi_\delta((\tilde x_i-x_i)/\varepsilon)}{0}
\delta h(t,\tilde x)\,d\tilde x\\&=
\frac{1}{\varepsilon^n}\int_{\R^n} 
\prod_{i=1}^n\psi_\delta((\tilde x_i-x_i)/\varepsilon)
\binom{-\nabla \delta h(t,\tilde x)}{0}\,d\tilde x.
\end{align*}
The remaining proof is identical to the one of Lemma \ref{lem:nueps}.
\end{proof}
We are now in the position to show the following result.
\begin{theorem}
If \eqref{Dhbnd} holds for the solution $(u,q)$ of \eqref{P} according to Theorem
\ref{thm:difforg} (which is satisfied for $\varepsilon_0>0$ small enough)
then it satisfies the VoF-type formulation
\eqref{weakVoF1}--\eqref{weakVoF3}.
\end{theorem}
\begin{proof}
Since the solution of \eqref{transpc} is $\alpha=1_{\Omega_1}(t)$ by
Proposition \ref{prop:alpha}, the formulations
\eqref{weakVoF1}--\eqref{weakVoF3} and \eqref{weak1}--\eqref{weak2} are equivalent if
the right hand side of \eqref{weakVoF1} coincides with the surface tension
force term \eqref{surface}. To show this we note that Lemma \ref{lem:normal}
yields for any $\varepsilon>0$
\begin{align*}
&-\int_{\R^{n+1}} \sigma
\frac{\nu_{\varepsilon}(t,x,y)^\top}{|\nu_{\varepsilon}(t,x,y)|}
 (D\varphi-\Div(\varphi)I)(x,y)
\nabla\alpha(t,x,y)\,d(x,y)=\\
&\int_{\R^n} \sigma
\frac{\nu_{\varepsilon}(t,x,h(t,x))^\top}{|\nu_{\varepsilon}(t,x,h(t,x))|}
 (D\varphi-\Div(\varphi)I)(x,h(t,x))\binom{-\nabla h(t,x)}{1}\,dx.
\end{align*}
Now the uniform convergence of $\nu_{\varepsilon}(t,x,h(t,x))$ to
$\binom{-\nabla h(t,x)}{1}$ for $\varepsilon\searrow 0$
by Lemma \ref{nueps} yields the convergence of the above term to \eqref{surface}.
\end{proof}
Finally, we can justify the following VoF-type formulation for computing the sensitivities
$(\delta u,\delta q)$. Due to the limited spatial regularity of $\partial_t
u$, we have to state time derivatives on the interface in weak form.
\begin{align}\notag
&\int_{J\times\R^{n+1}}
\bigl(\partial_t (\rho(\alpha) \delta u)
+\Div(\rho(\alpha) (\delta u\otimes u+u\otimes \delta u)+\delta c)^\top
\varphi\,d(t,x,y)\\\notag
&+\int_{J\times\R^{n+1}} S(\delta u,\delta
q;\mu(\alpha)):\nabla\varphi\,d(t,x,y)\\\notag
&+\int_{J}\int_{\R^{n+1}}
(\rho_2-\rho_1) u^\top\bigl(\partial_t\varphi+u\cdot\nabla\varphi\bigr) d\delta \alpha(t)(x,y)\,dt\\
&-\int_{J}\int_{\R^{n+1}} [S(u,q;\mu(\alpha))]:\nabla\varphi\,d\delta\alpha(t)(x,y)
dt=\label{weakVoF1lin}
\end{align}
\begin{align}
\notag
&=\lim_{\varepsilon\searrow 0}-\int_{J\times\R^{n+1}} \sigma
\left(\frac{\delta\nu_{\varepsilon}^\top}{|\nu_{\varepsilon}|}-
\frac{\delta\nu_{\varepsilon}^\top\nu_{\varepsilon}\nu_{\varepsilon}^\top}{|\nu_{\varepsilon}|^3}
\right)(D\varphi-\Div(\varphi)I)
\nabla\alpha\,d(t,x,y)\\\notag
&\quad -\int_{J\times\R^{n+1}} \sigma\frac{\nu_{\varepsilon}^\top}{|\nu_{\varepsilon}|}
(D\varphi-\Div(\varphi)I)
\nabla d\delta\alpha(t)(x,y)\qquad\forall\,\varphi\in C_c^2(J\times\R^{n+1};\R^{n+1}),\\
\label{weakVoF2lin}
&\int_{J\times\R^n} \Div(\delta u)\, \psi\,d(t,x,y)=0\quad\forall\,\psi\in 
C_c^1(J\times\R^{n+1}),\\
&\delta \alpha \mbox{ satisfies \eqref{transpclin}},\label{weakVoF3lin}
\end{align}
where $\nu_{\varepsilon}$ and $\delta\nu_{\varepsilon}$ are given by
\eqref{nueps} and \eqref{delnueps}.

We need the following Lemma
\begin{lemma}\label{lem:delalph}
Let $\psi\in C_c^1(\R^{n+1};\R^{n+1})$. Then
\begin{align*}
& -\int_{\R^{n+1}} \psi(x,y)^\top \nabla d\delta\alpha(t)(x,y)\\
&= \int_{\R^n} \partial_y \psi(x,h(t,x))^\top \binom{-\nabla
 h(t,x)}{1}\delta h(t,x)+\psi(x,h(t,x))^\top \binom{-\nabla \delta h(t,x)}{0}\,dx\\
&=\int_{\R^n} \Div(\psi)(x,h(t,x)) \delta h(t,x)\,dx.
\end{align*}
\end{lemma}
\begin{proof}
By the definition of distributional derivatives one has with
\eqref{delalpha}
\begin{align*}
&-\int_{\R^{n+1}} \psi(x,y)^\top \nabla d\alpha(t)(x,y)=
\int_{\R^{n+1}} \Div(\psi)(x,y) d\alpha(t)(x,y)\\
&=\int_{\R^n} \Div(\psi)(x,h(t,x)) \delta h(t,x)\,dx.
\end{align*}
On the other hand, integration by parts yields
\begin{align*}
& \int_{\R^n} \delta h(t,x)\partial_y \psi(x,h(t,x))^\top \binom{-\nabla
 h(t,x)}{1}+\psi(x,h(t,x))^\top \binom{-\nabla \delta h(t,x)}{0}\,dx\\
&=
 \int_{\R^n} \delta h(t,x)\partial_y \psi(x,h(t,x))^\top \binom{-\nabla
 h(t,x)}{1}\\
&\quad+\left(\sum_{i=1}^n \partial_{x_i}\psi_i(x,h(t,x))+
\partial_y \psi(x,h(t,x))^\top \binom{\nabla h(t,x)}{0}\right) \delta h(t,x)\,dx\\
&=\int_{\R^n} \Div(\psi)(x,h(t,x)) \delta h(t,x)\,dx.
\end{align*}
\end{proof}

\begin{theorem}
Let $(u,q)$ be the solution of \eqref{P} according to Theorem
\ref{thm:difforg} and let \eqref{Dhbnd} hold (which is satisfied for $\varepsilon_0>0$ small
enough). Moreover, let $(\delta u,\delta q)$ be the sensitivities
of $(u,q)$ corresponding to $(\delta u_0,\delta c)$. Then
$(\delta u,\delta q)$ solve the linearized VoF-type system
\eqref{weakVoF1lin}--\eqref{weakVoF3lin}.
\end{theorem}
\begin{proof}
Let $(u_0,h_0,c), (\delta u_0,0,\delta c)\in \U_u\times \U_h\times \U_c(h_0)$
be such that $(u_0,h_0,c)$ and $(u_0,h_0,c)+(\delta u_0,0,\delta c)$
satisfy the conditions of Theorem \ref{thm:difforg}.
Denote now by $(\hat u,\pi,[\pi],h)$ the unique solution of
\eqref{Ptrans} for data $(u_0,h_0,c)$ and by
$(\hat u^s,\pi^s,[\pi^s],h^s)$ the one
for data $(u_0,h_0,c)+s\,(\delta u_0,0,\delta c)$. Let
$(u,q)$ and $(u^s,q^s)$ be the corresponding states
in physical coordinates according to \eqref{trafo}  and let
$\alpha=1_{\Omega_1(t)}, \alpha^s=1_{\Omega_1^s(t)}$ be the corresponding
solutions of \eqref{transpc}. Finally, let
$(\delta u,\delta h,\delta q)$ be the directional derivatives
(sensitivities) in direction $(\delta u_0,0,\delta c)$ which exist by
Theorem \ref{thm:difforg}. By the differentiability result of
Theorem \ref{thm:difforg} we know that with the extensions
$u_\pm,q_\pm$ in \ref{upm}, see \eqref{udef}, \eqref{upmdef}, \eqref{qpmdef}
\begin{align}\label{udiff1}
 \frac{u^s-u}{s}&\to\delta u &&\mbox{in $C(\bar J;L_p(\R^{n+1};\R^{n+1}))$},\\
\label{udiff2}
 \frac{u_\pm^s-u_\pm}{s}&\to\delta u_\pm
&&\mbox{in $L_p(J;H_p^1(\R^{n+1};\R^{n+1}))$},\\
\label{qdiff}
 \frac{q_\pm^s-q_\pm}{s}&\to\delta q_\pm
&&\mbox{in $L_p(J;L_p(\R^{n+1};\R^{n+1}))$}.
\end{align}
We derive now the different terms in \eqref{weakVoF1lin}.
Let
\[
 \Omega_s=\{(t,x,y):\alpha^s=\alpha\},\quad
 \Omega_s^c=\{(t,x,y):\alpha^s(t)\ne\alpha\}
\]
We have for arbitrary $\varphi\in C_c^2(J\times \R^{n+1};\R^{n+1})$
\begin{align*}
&\int_{J\times\R^{n+1}}\!\!\!\frac{-1}{s}
\bigl((\rho(\alpha^s) u^s-\rho(\alpha) u)^\top \partial_t\varphi 
+\rho(\alpha^s) (u^s)^\top (u^s\cdot\nabla\varphi)
-\rho(\alpha) u^\top (u\cdot \nabla\varphi)\bigr)\,d(t,x,y)\\
&=\int_{\Omega_s}\frac{-1}{s}
\bigl(\rho(\alpha) (u^s-u)^\top \partial_t\varphi 
+\rho(\alpha) ((u^s)^\top (u^s\cdot\nabla\varphi)
-u^\top (u\cdot \nabla\varphi))\bigr)\,d(t,x,y)\\
&+\int_{\Omega_s^c}\frac{-1}{s}
\bigl((\rho(\alpha^s) u^s-\rho(\alpha) u)^\top \partial_t\varphi 
+\rho(\alpha^s)(u^s)^\top (u^s\cdot\nabla\varphi)
-\rho(\alpha) u^\top (u\cdot\nabla\varphi)\bigr)\,d(t,x,y)
\end{align*}
By \eqref{udiff1}, \eqref{udiff2} one obtains
\begin{align*}
&\int_{\Omega_s}\frac{-1}{s}
\bigl(\rho(\alpha) (u^s-u)^\top \partial_t\varphi 
+\rho(\alpha) ((u^s)^\top (u^s\cdot\nabla\varphi)
-u^\top (u\cdot \nabla\varphi))\bigr)\,d(t,x,y)\\
&\to\int_{J\times\R^{n+1}}-\bigl(\rho(\alpha) \delta u^\top \partial_t\varphi
+\rho(\alpha)(\delta u^\top (u\cdot\nabla\varphi)+u^\top (\delta u\cdot\nabla\varphi)
\bigr)\,d(t,x,y)\\
&=\int_{J\times\R^{n+1}}
\bigl(\partial_t (\rho(\alpha) \delta u)
+\Div(\rho(\alpha) (\delta u\otimes u+u\otimes \delta u))^\top
\varphi\,d(t,x,y).
\end{align*}
For the second summand we have by Theorem \ref{thm:difforg}
\begin{align*}
&\int_{\Omega_s^c}\frac{-1}{s}
\bigl((\rho(\alpha^s) u^s-\rho(\alpha) u)^\top \partial_t\varphi 
+\rho(\alpha^s)(u^s)^\top (u^s\cdot\nabla\varphi)
-\rho(\alpha) u^\top (u\cdot\nabla\varphi)\bigr)\,d(t,x,y)\\
&\int_{J\times \R^n}\!\!\!\frac{-1}{s}\int_{h(t,x)}^{\max(h(t,x),h^s(t,x))}
\!\!\!\!\!\!\!\!\!\!\!\!\!\!\!\!\!\!\!\!\!\bigl((\rho_1 u^s-\rho_2 u)^\top \partial_t\varphi 
+\rho_1(u^s)^\top (u^s\cdot\nabla\varphi)
-\rho_2 u^\top (u\cdot\nabla\varphi)\bigr)\,d(t,x,y)\\
&+\int_{J\times \R^n}\!\!\!\frac{-1}{s}\int_{h^s(t,x)}^{\max(h(t,x),h^s(t,x))}
\!\!\!\!\!\!\!\!\!\!\!\!\!\!\!\!\!\!\!\!\!\bigl((\rho_2 u^s-\rho_1 u)^\top \partial_t\varphi
+\rho_2(u^s)^\top (u^s\cdot \nabla\varphi)
-\rho_1 u^\top (u\cdot\nabla\varphi)\bigr)\,d(t,x,y)\\
&\to
\int_{J\times \R^n}
(\rho_2-\rho_1) u^\top\bigl(\partial_t\varphi+u\cdot\nabla\varphi
\bigr)(t,x,h(t,x))\delta h(t,x)\,d(t,x)\\
&=\int_{J}\int_{\R^{n+1}}
(\rho_2-\rho_1) u^\top\bigl(\partial_t\varphi+u\cdot\nabla\varphi
\bigr)
d\delta \alpha(t)(x,y)\,dt,
\end{align*}
where we have used \eqref{delalpha}, \eqref{imba} and \eqref{imbc} in the
last step.

For the next term in \eqref{weakVoF1lin} we note that
\begin{equation}\label{Suqconv}
\begin{split}
&\int_{J\times\R^{n+1}}
\frac{1}{s}(S(u^s,q^s;\mu(\alpha^s))-S(u,q;\mu(\alpha))):\nabla\varphi\bigr)\,d(t,x,y)\\
&=\int_{\Omega_s}
\frac{1}{s}(S(u^s-u,q^s-q;\mu(\alpha)):\nabla\varphi\,d(t,x,y)\\
&+\int_{\Omega_s^c} (S(u^s,q^s;\mu(\alpha^s))-
S(u,q;\mu(\alpha)):\nabla\varphi\,d(t,x,y).
\end{split}
\end{equation}
Now \eqref{udiff2}, \eqref{qdiff} yield
\begin{align*}
&\int_{\Omega_s}
\frac{1}{s}(S(u^s-u,q^s-q;\mu(\alpha)):\nabla\varphi\,d(t,x,y)
\to \int_{J\times\R^{n+1}} S(\delta u,\delta
q;\mu(\alpha)):\nabla\varphi\,d(t,x,y).
\end{align*}
Moreover, by using \eqref{imba} and Theorem \ref{thm:exdiff} we have
\begin{align*}
&\int_{\Omega_s^c} \frac{1}{s}(S(u^s,q^s;\mu(\alpha^s))-
S(u,q;\mu(\alpha)):\nabla\varphi\,d(t,x,y)\\
&=
\int_{J\times\R^n}\frac{1}{s}\int_{h(t,x)}^{\max(h(t,x),h^s(t,x))}
(S(u^s_-,q^s_-;\mu_1)-S(u_+,q_+;\mu_2):\nabla\varphi\,d(t,x,y)\\
&+\int_{J\times\R^n}\frac{1}{s}\int_{h^s(t,x)}^{\max(h(t,x),h^s(t,x))}
(S(u^s_+,q^s_+;\mu_2)-S(u_-,q_-;\mu_1)):\nabla\varphi\,d(t,x,y)\\
&\to -\int_{J\times\R^n} [S(u,q;\mu(\alpha))](t,x,h(t,x))\delta
h(t,x):\nabla\varphi(t,x,h(t,x))\,d(t,x)\\
&= -\int_{J}\int_{\R^{n+1}} [S(u,q;\mu(\alpha))]:\nabla\varphi\,d\delta\alpha(t)(x,y) dt.
\end{align*}
Here, we have used \eqref{delalpha} and \eqref{imba} in the
last step.

Finally, the surface tension term \eqref{surface}
has with the abbreviations
\[
  \tilde \nu(t,x)=\binom{-\nabla h(t,x)}{1},\quad \delta\tilde \nu(t,x)=\binom{-\nabla
\delta h(t,x)}{0}
\]
by Theorem \ref{thm:difforg}
and \eqref{imbc} the directional derivative
\begin{equation}\label{surfabl}
\begin{split}
&\int_{J\times\R^{n}} \sigma
\left(\frac{\delta\tilde\nu^\top}{|\tilde\nu|}-
\frac{\delta\tilde \nu^\top\tilde\nu\tilde\nu^\top}{|\tilde\nu|^3}
\right)(t,x)(D\varphi-\Div(\varphi)I)(t,x,h(t,x)) \tilde\nu(t,x)\,d(t,x)\\
&+\int_{J\times\R^{n+1}} \sigma\frac{\tilde\nu^\top}{|\tilde\nu|}
\Bigl(\partial_y(D\varphi-\Div(\varphi)I)(t,h(t,x)) \delta h(t,x)
\tilde\nu(t,x)\\
&\qquad\qquad +(D\varphi-\Div(\varphi)I)(t,h(t,x)) \delta \tilde\nu(t,x)\Bigr)\,d(t,x)
\end{split}
\end{equation}
Now the first integral on the right hand side of \eqref{weakVoF1lin}
converges to the first intergal in \eqref{surfabl} by first applying
Lemma \ref{lem:normal} and then Lemmas \ref{lem:nueps} and \ref{lem:delnueps}.
By using first Lemma \ref{lem:delalph} (note that $\nu_\varepsilon(t,x,y)$
depends close to $\Gamma(t)$ only on $x$ by \eqref{Dhbnd}, see \eqref{nueps}),
and then Lemma \ref{lem:nueps}
and the fact that $\nabla\delta h$ is continuous by \eqref{imbc},
the second integral on the right hand side of \eqref{weakVoF1lin}
converges to the second intergal in \eqref{surfabl}.

\eqref{weakVoF2lin} is obvious and \eqref{weakVoF3lin} follows by
Proposition \ref{prop:alpha}.
\end{proof}
\section{Analytical settings for the application of optimization
methods}\label{sec:4}
The results of this paper justify the application of derivative based
optimization methods. We discuss now some possible settings. While some are
canonical, the treatment of optimization problems involving the state across
the interface, in particular the pressure or
the position of the interface, requires care, since the pressure and phase indicator
field are discontinuous across the interface.

Let $p>n+3$ and $\U_{u}(h_0), \U_{c}(t_0)$ be defined by
\eqref{Udeforg}. Let $h_0\in \U_h$ with $\|h_0\|_{\U_h}$ small enough
\[
 U_{ad}:=\{ (u_0,c)\in \U_u(h_0)\times \U_{c}(t_0):
 (u_0,h_0,c) \mbox{ satisfy \eqref{comp}, \eqref{initepsorg}}\}
\]
and consider the control-to-state mapping
\[
 (u_0,c)\in U_{ad} \subset \U_u(h_0)\times \U_{c}(t_0)
 \mapsto (\hat u,\pi,[\pi],h)\in \E(t_0)
\]
given by \eqref{Ptrans},
which is differentiable at least for sufficiently small controls
by Theorem \ref{thm:difforg},
and the corresponding original state
\[
    (u,q)(t,x,y)=(\hat u,\pi)(1,x,y-h(t,x))
\]
solving \eqref{P}. Denote the solution operator by
\[
     (u,q)=({\mathcal S}^u(u_0,c),{\mathcal S}^q(u_0,c))={\mathcal S}(u_0,c).
\]
Now we are interested in optimization problems of the form
\[
  \min_{(u_0,c)} {\mathcal J}({\mathcal S}(u_0,c),u_0,c) \mbox{ s.t. }
  (u_0,c) \in U_{ad},
\]
where additional constraints would be possible. We discuss now analytical
settings, for which the continuous differentiability of the reduced
objective function $(u_0,c) \in U_{ad}\mapsto {\mathcal J}({\mathcal S}(u_0,c),u_0,c)$
is ensured.

\subsection{Objective functions involving the velocity field}\label{sec:41}
We consider first the simpler case, where the state-dependence of the
objective function involves only the velocity, i.e.,
\begin{equation}\label{Jredu}
(u_0,c) \in U_{ad}\mapsto {\mathcal J}({\mathcal S}(u_0,c),u_0,c)=
{\mathcal J}^u({\mathcal S}^u(u_0,c),u_0,c).
\end{equation}
Using the differentiability results of Theorem \ref{thm:difforg} with the space
given in \eqref{udef}, the reduced objective functional
\eqref{Jredu} is continuously differentiable, if the mapping
\[
 (u,u_0,c)\in C(\bar J; L_{\tilde p}(\R^{n+1},\R^{n+1})\times
 \U_u(h_0)\times \U_{c}(t_0) \mapsto {\mathcal J}^u(u,u_0,c)
\]
is continuously differentiable for some $\tilde p\in [p,\infty)$.
This applies for example in the case of least squares functionals or many
other types of tracking functionals. The derivative is easily obtained
by the chain rule and by using the sensitivities $\delta u$, where
$\delta u$ can be obtained by the VoF-type formulation
\eqref{weakVoF1lin} or by using the linearization of the transformed
problem \eqref{Ptrans} together with Lemma
\ref{lem:uqdiff} and Lemma \ref{lem:cdiff}.

If the objective function \eqref{Jredu} evaluates the velocity field $u$
only in an open observation domain $J\times \Omega_o$ with positive distance from
the interface $\bigcup_{t\in J} (\{t\}\times \Gamma(t))$, then 
it is by Theorem \ref{thm:difforg} sufficient if the mapping
\[
 (u,u_0,c)\in (C(\bar J; L_{\tilde p}(\R^{n+1},\R^{n+1})\cap
L_p(J;H_p^1(\Omega_o,\R^{n+1})))\times
 \U_u(h_0)\times \U_{c}(t_0) \mapsto {\mathcal J}^u(u,u_0,c)
\]
is continuously differentiable for some $\tilde p\in [p,\infty)$.

\subsection{Objective functions involving the pressure or phase
indicator}\label{sec:42}
Since the pressure field is discontinuous across the
interface, the differentiability results of Theorem \ref{thm:difforg}
apply only for extensions $q_\pm$ of the pressure across the interface.
We show now that certain types of objective functionals are nevertheless
differentiable.

We consider objective functions of the form
\begin{equation}\label{Jredq}
\begin{split}
(u_0,c) \in U_{ad} &\mapsto {\mathcal J}({\mathcal S}(u_0,c),u_0,c)\\
&={\mathcal J}^q({\mathcal S}^q(u_0,c),u_0,c)
=\int_J \int_{R^{n+1}} \ell({\mathcal S}^q(u_0,c)(t,x,y),x,y)\,d(x,y)\,dt.
\end{split}
\end{equation}
Here, $\ell: \R\times \R^n\times\R \to \R$ is a continuously differentiable
function such that there are $R>0, C_\ell>0$ with (other settings are
possible)
\begin{equation}\label{ellprop}
\begin{split}
  &\operatorname{supp}(\ell(q,\cdot))\subset \{(x,y)\in \R^n\times\R:\, |(x,y)|\le R\}\quad\forall\,q\in \R,\\
&  |\partial^2_{(q,y)} \ell(q,x,y)|\le C_\ell (1+|q|^{p-2})\quad\forall\,(x,y)\in \R^n\times\R.
\end{split}
\end{equation}
This implies the estimate
\begin{align*}
|\partial_q\ell(q,x,y)|&\le |\partial_q\ell(0,x,y)|+\left|\int_0^q
\partial^2_{qq} \ell(s,x,y)\,ds\right|
\le |\partial_q\ell(0,x,y)|+C_\ell (|q|+|q|^{p-1})\\
&\le C_\ell' (1+|q|^{p-1})
\end{align*}
with the constant $C_\ell'=\|\partial_q\ell(0,\cdot)\|_{L_\infty(\R^{n+1})}+2 C_\ell$
as well as
\begin{align}\label{ellprop3}
|\ell(q,x,y)|\le |\ell(0,x,y)|+\left|\int_0^q \partial_q \ell(s,x,y)\,ds\right|
\le C_\ell'' (1+|q|^{p})
\end{align}
with $C_\ell''=\|\ell(0,\cdot)\|_{L_\infty(\R^{n+1})}+2 C_\ell'$.

Now, using the transformation $(\hat x,\hat y)=T_{h(t)}(x,y)=(x,y-h(t,x))$
we can reqrite \eqref{Jredq} as
\begin{equation}\label{Jqvar}
\begin{split}
{\mathcal J}^q(q,u_0,c)&=
\int_J \int_{R^{n+1}} \ell(q(t,x,y),x,y)\,d(x,y)\,dt\\
&=
\int_J \int_{R^{n+1}} \ell(\pi(t,x,y-h(t,x)),x,y)\,d(x,y)\,dt\\
&=\int_J \int_{R^{n+1}} \ell(\pi(t,\hat x,\hat y),\hat x,\hat y+h(t,\hat
x))\,d(\hat x,\hat y)\,dt\\
&=\int_J \int_{\dot R^{n+1}} \ell(\pi(t,\hat x,\hat y),\hat x,\hat y+h(t,\hat
x))\,d(\hat x,\hat y)\,dt.
\end{split}
\end{equation}
We will now show the following result.
\begin{theorem}\label{thm:Jqdiff}
Let \eqref{ellprop} hold. Then with the convention \eqref{prep}
the objective function \eqref{Jredq} is continuously differentiable with
derivative
\begin{align*}
({\mathcal J}^q)'({\mathcal S}^q(u_0,c),u_0,c)\cdot \binom{\delta u_0}{\delta c}
&=\int_J \int_{\dot R^{n+1}} 
\partial_q\ell(q(t,x,y),x,y) \delta q(t,x,y)\, d(x,y)\,dt\\
&-\int_J \int_{R^{n}} [\ell(q(t,x,h(t,x)),x,h(t,x))] \delta h(t,x)\,dx\,dt\\
&=\int_J \int_{\dot R^{n+1}} 
\partial_q\ell(q(t,x,y),x,y) \delta q(t,x,y)\, d(x,y)\,dt\\
&-\int_J \int_{R^{n}} [\ell(q(t,x,h(t,x)),x,h(t,x))] d\delta\alpha(t)(x,y)\,dt,
\end{align*}
where $[\ell(q(t,x,h(t,x)),x,h(t,x))]$ is the jump across $\Gamma(t)$ at
$(x,h(t,x))$ and
the sensitivities $\delta q$, $\delta \alpha$
can be obtained by the VoF-type formulation
\eqref{weakVoF1lin} or $\delta \pi$, $\delta h$ can be computed
by using the linearization of the transformed
problem \eqref{Ptrans} together with Lemma \ref{lem:cdiff}
and $\delta q$ by using Lemma \ref{lem:uqdiff}.
\end{theorem}
\begin{remark}
If we consider an objective function ${\mathcal J}^\alpha$ of the form
\eqref{Jredq} with the pressure $q$ replaced by the phase indicator $\alpha$
then an analogue of Theorem \ref{thm:Jqdiff} holds and the derivative
simplifies to
\begin{align*}
({\mathcal J}^\alpha)'({\mathcal S}^\alpha(u_0,c),u_0,c)\cdot \binom{\delta u_0}{\delta c}
&=-\int_J \int_{R^{n}} [\ell(\alpha(t,x,h(t,x)),x,h(t,x))] d\delta\alpha(t)(x,y)\,dt,
\end{align*}
since $\delta\alpha$ has its support on the complement of $\dot\R^{n+1}$.
\end{remark}
We use the following auxiliary result.
\begin{lemma}\label{lem:Jq}
Let \eqref{ellprop} hold. Then the mapping
\begin{align*}
&(\pi,h)\in L_p(J;L_\infty(\R^{n+1}))\times L_p(J;L_\infty(\R^{n}))\\
&\mapsto \widehat{\mathcal J}^q(\pi,h):=
\int_J \int_{\dot R^{n+1}} \ell(\pi(t,\hat x,\hat y),\hat x,\hat y+h(t,\hat
x))\,d(\hat x,\hat y)\,dt
\end{align*}
is continuously differentiable with derivative
\begin{align*}
 (\widehat{\mathcal J}^q)'(\pi,h)\cdot \binom{\delta\pi}{\delta h}
=\int_J \int_{\dot R^{n+1}} \binom{\partial_q\ell}{\partial_y\ell}(\pi(t,\hat x,\hat y),\hat x,
\hat y+h(t,\hat x))^\top \binom{\delta \pi(t,\hat x,\hat y)}{\delta h(t,\hat x)}\,d(\hat x,\hat y)\,dt.
\end{align*}
\end{lemma}
\begin{proof}
We have by using \eqref{ellprop} and its consequence \eqref{ellprop3}
\begin{align*}
 \widehat{\mathcal J}^q(\pi,h)&\le
\int_J \int_{\dot R^{n+1}} |\ell(\pi(t,x,y-h(t,x)),x,y)|\,d(x,y)\,dt\\
&\le \int_J \int_{|(x,y|\le R} C_\ell'' (1+|\pi(t,x,y-h(t,x))|^p)\,d(x,y)\,dt\\
&\le C_\ell'' (2R)^{n+1} \int_J (1+\|\pi(t,\cdot)\|_{L_\infty(\R^{n+1})}^p)
\le C_\ell'' (2R)^{n+1} (t_0+\|\pi\|_{L_p(J;L_\infty(\R^{n+1}))}^p).
\end{align*}
%
To show the differentiability, we note that by Taylor expansion 
and \eqref{ellprop3}
\begin{align*}
&\Biggl|\ell((\pi+\delta\pi)(t,\hat x,\hat y),\hat x,\hat y+(h+\delta h)(t,\hat x)))
-\ell(\pi(t,\hat x,\hat y),\hat x,\hat y+h(t,\hat x)))\\
&-\binom{\partial_q\ell}{\partial_y\ell}(\pi(t,\hat x,\hat y),\hat x,
\hat y+h(t,\hat x))^\top \binom{\delta \pi(t,\hat x,\hat y)}{\delta h(t,\hat
x)}\Biggr|\\
&\le C_\ell (1+(|\pi(t,\hat x,\hat y)|+|\delta \pi(t,\hat x,\hat y)|)^{p-2})
(|\delta \pi(t,\hat x,\hat y)|^2+|\delta h(t,\hat x)|^2).
\end{align*}
Hence, we obtain with the asserted derivative
\begin{align*}
&\left|\widehat{\mathcal J}^q(\pi+\delta\pi,h+\delta h)-\widehat{\mathcal J}^q(\pi,h)
-(\widehat{\mathcal J}^q)'(\pi,h)\cdot \binom{\delta\pi}{\delta h}\right|\\
&\le \int_J \int_{\dot\R^{n+1}}
C_\ell (1+(|\pi(t,\hat x,\hat y)|+|\delta \pi(t,\hat x,\hat y)|)^{p-2})
(|\delta \pi(t,\hat x,\hat y)|^2+|\delta h(t,\hat x)|^2)\,d(\hat x,\hat y)\,dt\\
&\le (2R)^{n+1} C_\ell \int_J (1+\||\pi(t,\cdot)|+|\delta\pi(t,\cdot)|\|_{L_\infty}^{p-2})
\cdot(\|\delta \pi(t,\cdot)\|_{L_\infty}^2+
\|\delta h(t,\cdot)\|_{L_\infty}^2)\,dt\\
&\le
(2R)^{n+1} C_\ell(t_0^{(p-2)/p}+\||\pi|+|\delta\pi|\|_{L_p(J;L_\infty)}^{p-2})
(\|\delta \pi\|_{L_p(J;L_\infty)}^2+
\|\delta h\|_{L_p(J;L_\infty)}^2)\\
&=O((\|\delta \pi\|_{L_p(J;L_\infty)}^2+\|\delta h\|_{L_p(J;L_\infty)}^2)),
\end{align*}
where we have applied H\"older's inequality in the last step.
The continuity of the derivative follows with \eqref{ellprop3}
by very similar calculations.
\end{proof}
\begin{proof} (of Theorem \ref{thm:Jqdiff})
From Theorem \ref{thm:exdiff}, Theorem \ref{thm:embed} and Lemma \ref{lem:cdiff} we know that
\[
(u_0,c) \in U_{ad}\mapsto (\hat u,\pi,[\pi],h) \in \E(t_0)
\]
is continuously differentiable and with the convention \eqref{prep}
the mapping $(\hat u,\pi,[\pi],h) \in \E(t_0) \mapsto \pi\in
L_p(J;H_p^1(\dot\R^{n+1}))$ is linear and continuous by
\eqref{ppoin}. Moreover, by $p>n+3$ we have clearly
$L_p(J;H_p^1(\dot\R^{n+1}))\hookrightarrow L_p(J;L_\infty(\R^{n+1}))$
and $\E_4(t_0) \hookrightarrow L_p(J;L_\infty(\R^{n}))$. Hence,
the mapping
\[
(u_0,c) \in U_{ad}\mapsto (\pi,h) \in L_p(J;L_\infty(\R^{n+1}))\times L_p(J;L_\infty(\R^{n}))
\]
is continuously differentiable and thus \eqref{Jredq} is continuously differentiable
by Lemma \ref{lem:Jq} and the last representation of \eqref{Jredq}
in \eqref{Jqvar}. The derivative is given by Lemma \ref{lem:Jq}.
Using that
\begin{align*}
\frac{d}{d\hat y} \ell(\pi(t,\hat x,\hat y),\hat x,\hat y+h(t,\hat x))
&=\partial_y\ell(\pi(t,\hat x,\hat y),\hat x,\hat y+h(t,\hat x))\\
&+\partial_q\ell(\pi(t,\hat x,\hat y),\hat x,\hat y+h(t,\hat x))
\partial_{\hat y}\pi(t,\hat x,\hat y),
\end{align*}
integration by parts yields
\begin{align*}
& (\widehat{\mathcal J}^q)'(\pi,h)\cdot \binom{\delta\pi}{\delta h}
=\int_J \int_{\dot R^{n+1}} \binom{\partial_q\ell}{\partial_y\ell}(\pi(t,\hat x,\hat y),\hat x,
\hat y+h(t,\hat x))^\top \binom{\delta \pi(t,\hat x,\hat y)}{\delta h(t,\hat x)}\,d(\hat x,\hat y)\,dt\\
&=\int_J \int_{\dot R^{n+1}} 
\partial_q\ell(\pi(t,\hat x,\hat y),\hat x,\hat y+h(t,\hat x))
(-\partial_{\hat y}\pi(t,\hat x,\hat y) \delta h(t,\hat x)+
\delta\pi(t,\hat x,\hat y))\, d(\hat x,\hat y)\,dt\\
&-\int_J \int_{R^{n}} [\ell(\pi(t,\hat x,0),\hat x,h(t,\hat x))] \delta h(t,\hat x)\,d\hat x\,dt\\
&=\int_J \int_{R^{n+1}} 
\partial_q\ell(q(t,x,y),x,y) \delta q(t,x,y)\, d(\hat x,\hat y)\,dt\\
&-\int_J \int_{R^{n}} [\ell(q(t,x,h(t,x)),x,h(t,x))] \delta h(t,x)\,dx\,dt.
\end{align*}
Here, we have used Lemma \ref{lem:uqdiff} and
\eqref{delalpha} in the last step.
\end{proof}

\subsection*{Acknowledgments}
The work of Johannes
Haubner and Michael Ulbrich was supported by the Deutsche
Forschungsgemeinschaft (DFG, German Research Foundation) -- Project Number
188264188/GRK1754 -- as part of the International Research Training Group IGDK
1754 Optimization and Numerical Analysis for Partial Differential Equations
with Nonsmooth Structures. The work of Elisabeth Diehl and Stefan Ulbrich
was funded by the Deutsche Forschungsgemeinschaft (DFG, German Research Foundation)
-- Project-ID 265191195 -- SFB 1194 Interaction between Transport and Wetting Processes,
Project B04.
\bibliography{ref}
\bibliographystyle{abbrv}
\end{document}